\documentclass[11pt]{article}

\pdfoutput=1

\RequirePackage{my_packages}
\RequirePackage{my_theorems}
\RequirePackage{my_macros}

\usepackage[backend = biber,        
            language = english ,
            style = alphabetic-verb ,   
            giveninits = true,
            isbn = false,
            url = false,
            doi = false,
            sorting = nyt,
            maxnames = 6,
            backref=true
            ]{biblatex}

\renewbibmacro{in:}{%
  \ifentrytype{article}{}{%
  \printtext{\bibstring{in}\intitlepunct}}}
\AtEveryBibitem{%
  \clearfield{note}%
  \clearfield{edition}%
  \clearlist{language}
}

\DeclareNameAlias{labelname}{given-family}

\title{Free-abelian by free groups: homomorphisms and algorithmic explorations
}
\author{André Carvalho$^*$, Jordi Delgado$^\dag$}
\date{\vspace{-7pt}
    $^*$Centro de Matemática, Faculdade de Ciências da Universidade do Porto\\%
    $^\dag$Departament de Matemàtiques, Universitat Politècnica de Catalunya\\[3ex]%
    \today
}

\newcommand{\Addresses}{{
  \bigskip
  \footnotesize

  André Carvalho\\\nopagebreak
  \textsc{Centro de Matemática, Faculdade de Ciências da Universidade do Porto\\
  R. Campo Alegre s/n\\ 4169-007 Porto, Portugal}\\\nopagebreak
  \url{andrecruzcarvalho@gmail.com}

  \medskip

  Jordi Delgado\\\nopagebreak
  \textsc{Departament de Matemàtiques, Universitat Politècnica de Catalunya}\\\nopagebreak
  \url{jorge.delgado@upc.edu }

}}
\hypersetup{final}  

\begin{document}

\maketitle

 \begin{abstract}
 \noindent 
 We obtain an explicit description of the endomorphisms of free-abelian by free groups ($\Fn \ltimes \Zm$) together with a characterization of when they are injective and surjective. As a consequence we see that free-abelian by free groups are Hopfian and not coHopfian, and we investigate the isomorphism problem and the Brinkmann Problem for this family of groups. In particular, we prove that the isomorphism problem (undecidable in general) is decidable when restricted to finite actions, and that the Brinkmann Problem is decidable both for monomorphisms and automorphisms.
 \end{abstract}

 \bigskip



\vspace{15pt}


Free-abelian by free groups~(FABFs), namely groups of the form $\Fn \ltimes \Zm$ ($n \geq 2$) constitute an intriguing family which has aroused interest in recent times, both from the algebraic and the algorithmic points of view. Even the apparently tame case of free-abelian times free groups~(FATFs) $\Fn \times \Zm$ (corresponding to the trivial action) turns out to present nontrivial behaviors that have given rise to several works studying, for example, subgroup intersections~\parencite{Delgado_stallings_FTA_2022}, relative order \parencite{delgado_relative_2022}, compression and inertia \cite{roy_degrees_2021}, or different aspects of its group of automorphisms~\parencite{delgado_algorithmic_2013,roy_fixed_2020,carvalho_dynamics_2022}. 

The case of general actions is, of course, much more complicated, and many natural problems on it are still open, or even known to be algorithmically undecidable. For example, in~\parencite{delgado_extensions_2017}, a variation of Stallings automata was developed in order to solve the subgroup membership problem (\MP) for FABFs. However, the natural approach of considering products of automata to solve the subgroup intersection problem~(\SIP) was only successful for trivial actions (namely for FATFs) leaving $\SIP$ open for $\FABF$s. Even more indicative of the aforementioned complicatedness, in \parencite{bogopolski_orbit_2010}, \citeauthor{bogopolski_orbit_2010} discover a neat characterization for the decidability of the conjugacy problem (\CP) of certain group extensions which naturally provides the undecidability of $\CP(\FABF)$, and from this --- after an astute elaboration --- the undecidability of the isomorphism problem $\IP(\FABF)$ (see the unpublished note~\parencite{levitt_unsolvability_2008}).

In \Cref{sec: FABFs} we present free-abelian by free groups (FABFs), state some elementary facts about them, and set the basic terminology and notation used in the paper.

In \Cref{sec: homos} we 
 address the general study of homomorphisms between FABFs for which we obtain explicit descriptions that classify them in two main types, 
\emph{type I} (containing all the injective ones), and a degenerated case that we call \emph{type II}. 
In particular, we see that the automorphisms of $\Fn \ltimes \Zm$ are given by triples $(\varphi,\matr{Q},\matr{P})$, where $\varphi \in \Aut(\Fn)$, $\matr{Q} \in \Aut (\Zm) = \GL_m(\ZZ)$, and $\matr{P}$ is an $n \times m$ integer matrix linking the free and the free-abelian part. As a consequence, we obtain the Hopfianity and non-coHopfianity of these groups, and we deduce expressions for powers of endomorphisms.

In \Cref{sec: IP} we explore the isomorphism problem 
(known to be undecidable for general FABF groups)
for the cases where the kernel of the action is finitely generated: in case the action has finite image we can prove the decidability of the isomorphism problem, and in case the action is faithful
we prove that the isomorphism problem is equivalent to a constrained version of the subgroup conjugacy problem in $\GLm$.

In \Cref{sec: orbit problems} we deal with the Brinkmann orbit problem (given $g,h \in G$ and $\varphi \in \End(G)$, does there exist some $n\geq 0$ such that $(g)\varphi^{n} = h$?) which we prove to be solvable for endomorphisms of type I (including monomorphisms and epimorphisms) after addressing some general considerations.

\subsection*{Notation and terminology}
We summary below some general notation, terminology, and conventions used throughout the paper; more specific notation and terminology is introduced in the parts of the paper where it is first used.

We use lowercase boldface Latin font (${\vect{a},\vect{b},\vect{c},\ldots}$) to denote elements in $\Zm$, and
uppercase boldface Latin font to denote matrices ($\matr{A},\matr{B},\matr{C},\ldots $).
Lowercase Greek letters ($\auti_i,\varphi, \chi,\ldots$) are used to denote homomorphisms involving the (typically free or free-abelian) factor groups, whereas uppercase Greek letters ($\Phi,\Psi,\Omega,\cdots$) are used to denote homomorphisms between semidirect products.  
Furthermore,
homomorphisms and matrices are assumed to act on the right;
that is, we denote by $(x)\varphi$ (or simply $x \varphi$) the image of the element $x$ by the homomorphism $\varphi$, and we denote by  $\varphi \psi$ the composition
\smash{$A \overset{\varphi}{\to} B \overset{\psi}{\to} C$}.

 Regarding computability, we denote by
\begin{equation*}
\mathsf{P} \,\equiv\,
\gen{\texttt{cond}}\?_{\gen{\texttt{input}}}
\end{equation*}
the decision problem $\mathsf{P}$ with inputs in $\gen{\texttt{input}}$ and outputs \yep/\nop\ depending on whether the condition $\gen{\texttt{cond}}$ holds. For example, if $G = \pres{X}{R}$ is  a finitely presented group, then,
\begin{itemize}
    \item the word problem for $G$ is $\WP(G) \equiv
        w\in\normalcl{R}\?_{w \in (X^{\pm})^*}$ 
    \item the conjugacy problem for $G$ is $\CP(G) \equiv \exists w \in (X^{\pm})^* \st w^{-1} u w =_G v \,\?\,_{u,v \in (X^{\pm})^*}$
\end{itemize}


Also, if $\mathcal{G}$ is a family of groups, and $\mathsf{P}(G)$ is a decision problem on $G$, we 
say that $\mathsf{P}$ is decidable in $\mathcal{G}$ (or that $\mathsf{P}(\mathcal{G})$ is decidable) if $\mathsf{P}(G)$ is decidable (\ie there exists a Turing machine computing it) for every $G \in \mathcal{G}$; and that $\mathsf{P}$ is \defin{undecidable in $\mathcal{G}$} (or that $\mathsf{P}(\mathcal{G})$ is undecidable) otherwise. For example, if we denote by $\mathcal{H}$ the class of hyperbolic groups, we might say that the isomorphism problem $\IP(\mathcal{H})$ is decidable (see~\cite{dahmani_isomorphism_2011}), and that the subgroup membership problem $\MP(\mathcal{H})$ is undecidable (see~\cite{rips_subgroups_1982}).

\section{Free-abelian by free groups} \label{sec: FABFs}
We call \emph{free-abelian by free} (FABF) the groups $G$ that fit in the middle of a short exact sequence of the form 
\[\trivial \to \Zm \to G \to \Fn \to \trivial \,,\]
where $\Zm$ is a \emph{nontrivial} f.g.~free-abelian group (\ie $m \geq 1$), and $\Fn$ is a \emph{noncyclic} f.g.~free group (\ie $n \geq 2$).\footnote{We intentionally exclude from the family the degenerated cases corresponding to 
free and free-abelian groups, and the groups of the form $\ZZ \ltimes \ZZ^m
$, which are polycyclic and therefore well known and very different in nature from the rest of the family.}
Since the quotient group $\Fn$ is free, the given short exact sequence always splits, and $G$ can be seen as a semidirect product of the form $\Fn \ltimes \Zm$.

Equivalently, a group is \defin{free-abelian by free} (FABF) if
it admits a presentation of the form
\begin{equation} \label{eq: presentation ABF}
    \Ga
    \,=\,
    \ABFa
    \,=\,
    \Pres{\hspace{-3pt}
    \begin{array}{ll}
    t_1,\ldots,t_m\\
    x_1,\ldots,x_n
    \end{array}
    }
    {
    \begin{array}{ll}
    t_i^{-1} t_j t_i = t_j & \forall i,j \in [1,m]\\
    x_i^{-1} t_j x_i = (t_j)\auti_i & \forall i \in [1,n],\forall j\in[1,m]
    \end{array} \hspace{-5pt}
    }
\end{equation}
where $m \geq 1$, $n \geq 2$, and
%
$\acti \colon \Fn \to \Aut(\Zm)$ is a homomorphism
$x_i \mapsto \auti_i$
($i \in [1,n])$.
Note that then, $\gen{t_1,\ldots,t_m}_{\Ga} \isom \ZZ^m$, $\gen{x_1, \ldots, x_n}_{\Ga} \isom \Fn$,
and by~$(t_j)\auti_i$ in \eqref{eq: presentation ABF} we mean (a word in $\set{t_1,\ldots,t_m}^{\pm}$ representing) the image under $\auti_i$ of
(the element in $\ZZ^m$ represented by)~$t_j$. We denote by $\FABF$ the class of free-abelian by free groups.

Given a word $u \in \Fn$, we write
$\auti_u = (u)\acti$,
and we generalize this notation to arbitrary subsets $S \subseteq \Fn$ in the natural way, \ie $\auti_S = (S)\acti$.
Also, for every $\vect{a} = (a_1,\ldots,a_m) \in \Zm$, we abbreviate
$t_1^{a_1} t_2^{a_2} \cdots t_m^{a_m} = \t^{\vect{a}}$; in particular, ${t_j = \t^{\vect{e_j}}}$, where $\vect{e_j}$ denotes the $j$-th canonical generator in~$\Zm$. It is clear that~$\set{\fba{u}{a} \colon \vect{a}\in\Zm, u \in \Fn}$ is a set of normal forms for the elements in~$\Ga$ (corresponding to the set of pairs $\set{(u,\vect{a}) \st u \in \Fn \text{ and } \vect{a}\in \Zm }$) and hence the word problem $\WP(\Ga)$ is decidable for every \FABF\ group $\Ga$. In contrast, the conjugacy problem (\CP) has been shown to be not always decidable within this family of groups (as proved by \citeauthor{bogopolski_orbit_2010} in \cite{bogopolski_orbit_2010}).

We denote by $\tau \colon \Ga \to \ZZ^m, \fba{u}{a} \mapsto t^{\vect{a}}$ and $\pi\colon \Ga \to \Free[X], \fba{u}{a} \mapsto u$ the projections to the free-abelian and free part respectively.
We usually identify
the automorphisms of $\Zm = \gen{t_1,\ldots,t_m}$ with their corresponding integer (in fact unimodular) matrices \wrt the canonical basis $(\vect{e_1},\ldots,\vect{e_m})$. We write $\matr{A_i} = M(\auti_i,(\vect{e_j})_j)$,
for every $i \in [1,n]$;
$\matr{A}_u = M(\auti_u,(\vect{e_j})_j)$ for every $u \in \Fn$;
and
$\matr{A}_S = \set{\matr{A}_u \st u \in S}$, for every $S \subseteq \Fn$.
Of course, $\acti$ being a homomorphism means that for every $u,v \in \Fn$, $\matr{A_{uv}} = \matr{A_u} \matr{A_v}$ and $\matr{A_{u^{-1}}} = \matr{A_u^{-1}}$.
Then, the presentation \eqref{eq: presentation ABF} takes the form
\begin{equation} \label{eq: presentation ABF2}
    \Ga
    \,=\,
    \ABFa
    \,=\,
    \Pres{\hspace{-3pt}
    \begin{array}{ll}
    t_1,\ldots,t_m\\
    x_1,\ldots,x_n
    \end{array}
    }
    {
    \begin{array}{ll}
    t_i^{-1} t_j t_i = t_j & \forall i,j \in [1,m]\\
    x_i^{-1} t_j x_i = \t^{\vect{a_{ij}}} & \forall i \in [1,n],\forall j\in[1,m]
    \end{array} \hspace{-5pt}
    }
\end{equation}
where $\vect{a_{ij}} = \vect{e_j} \matr{A_i}$ is the $j$-th row of the matrix $\matr{A_i}$. 

We will assume that FABF groups are given in this form; that is, a group $\Ga$ is characterized by the following `data': two  integers $m\geq 1$, $n \geq 2$, and an $n$-tuple of
unimodular ${m \times m}$
matrices $\matr{A_1},\ldots,\matr{A_n} \in \GL_m(\ZZ)$.
Note that the relations in \eqref{eq: presentation ABF2} lead to the following ``multiplication rules'' in $\Ga$: for every $u \in \Fn$ and every $\vect{a} \in \ZZ^m$, $u^{-1} \t^{\vect{a}} u = \t^{\vect{a}\matr{A_u}}$, and hence
\begin{equation}
\ta{a} u
\,=\,
u \t^{\vect{a} \matr{A}_u}
\hspace{20pt}
\text{and}
\hspace{20pt}
 u \t^{\vect{a}}
\,=\,
\t^{\vect{a} \matr{A}_{u}^{-1}} u \,.
\end{equation}
That is, abelian elements are allowed to jump right (\resp left) over free elements at the price of being affected by (\resp the inverse of) the automorphism associated to the jumped free element. In particular we have the following expressions for inverses and conjugates in $\Ga$.

\begin{lem}
Let $\fba{u}{a},\fba{w}{c} \in \Ga$. Then:
\begin{enumerate}[ind]
    \item $(\fba{u}{a})^{-1}
    = u^{-1} \, \t^{- \vect{a} \matr{A}_{u}^{-1}}
    = u^{-1} \, \t^{\, \vect{a} (-\matr{A}_{u}^{-1})}
    $;
    \item $(\fba{w}{c})^{-1} (\fba{u}{a}) (\fba{w}{c})
    =  w^{-1} u w \ 
    \t^{\vect{c}(\matrid -  \matr{A}_{w^{-1} u w}) + \vect{a} \matr{A}_w}$.
\end{enumerate}
\end{lem}

Below we describe arbitrary products in $\Ga$, and introduce some notation to abbreviate their (somewhat convoluted) form.

\begin{nott}
Let $\Ga$ be a FABF group, let $w = y_{i_1}^{\eps_{1}} y_{i_2}^{\eps_{2}} \cdots y_{i_l}^{\eps_{l}}  \in \Free[Y] = \pres{y_1,\ldots,y_k}{-}$ be a formal nontrivial free word
on~$Y = \set{y_1,\ldots,y_k}$,
let $g_1,\ldots,g_k\in G$, and let $i,j \in [1,l]$ where $i<j$. Then, we denote by:
\begin{itemize}
    \item $w[j]$ the $j$-th letter (\ie monic subword) in $w$
    (hence, $w[j] = y_{i_j}^{\eps_j}$),
    \item $w[i,j]$ the subword $w[i]\cdots w[j]$ of $w$ (of length $j-i+1$),
    \item $ 
w[g_1,\ldots, g_k]
$
the $l$-tuple obtained after replacing
$y_j$ by $g_j$ in $w$,
    \item $w(g_1,\ldots, g_k)
$
the result of evaluating $w[g_1,\ldots, g_k]$ in $\Ga$.
\end{itemize}
We usually abbreviate $\overrightarrow{g_i} = (g_1,\ldots,g_k)$.
Note that then $w[\overrightarrow{g_i}] = w[g_1,\ldots, g_k] \in (\Ga)
^l$ whereas
$w(\overrightarrow{g_i}) = w(g_1,\ldots, g_k)
\in \Ga$.

\end{nott}
If $\fba{u_1}{p_1},
\ldots, \fba{u_k}{p_k} \in \Ga$
then,
\begin{align*}
w(u_1 \ta{p_1},\ldots, u_k \ta{p_k}) 
&\,=\,
(u_{i_1} \ta{p_{i_1}})^{\eps_1}
(u_{i_2} \ta{p_{i_2}})^{\eps_2}
\cdots
(u_{i_l} \ta{p_{i_l}})^{\eps_l}
\\
&\,=\,
(u_{i_1}^{\eps_1} \t[\widetilde{p}_{i_1}])
(u_{i_2}^{\eps_2} \t[\widetilde{p}_{i_2}])
\cdots
(u_{i_l}^{\eps_l} \t[\widetilde{p}_{i_l}]),
\\
&\,=\,
\t^{\,\vect{\widetilde{p}_{i_1}}
\matr{A}_{w[1](\overrightarrow{u_i})}^{-1} + 
\vect{\widetilde{p}_{i_2}}
\matr{A}_{w[1,2](\overrightarrow{u_i})}^{-1} 
+ \cdots +
\vect{\widetilde{p}_{i_l}}
\matr{A}_{w(\overrightarrow{u_i})}^{-1}
} \,  w(\overrightarrow{u_i})\\
&\,=\,  \t^{\sum_{j=1}^{l}\vect{\widetilde{p}_{i_j}}\,
\matr{A}_{w[1,j](\overrightarrow{u_i})}^{-1}
} \,  w(\overrightarrow{u_i})\\
&\,=\,  w(\overrightarrow{u_i}) \, \t^{\sum_{j=1}^{l}\vect{\widetilde{p}_{i_j}}\,
\matr{A}_{w[1,j](\overrightarrow{u_i})}^{-1}
\,
\matr{A}_{w(\overrightarrow{u_i})}
} \\
&\,=\,  w(\overrightarrow{u_i}) \, \t^{\sum_{j=1}^{l}\vect{\widetilde{p}_{i_j}}\,
\matr{A}_{w[j+1,l](\overrightarrow{u_i})}
} \\
&\,=\,  w(\overrightarrow{u_i}) \, \t^{\,\widetilde{\vect{p}}_{w}\,
\widetilde{\matr{A}}_{w(\overrightarrow{u_i})}} \, ,
\end{align*}
where:
\begin{itemize}
    \item we write $(u_{i_j} \ta{p_{i_j}})^{\eps_j} =  u_{i_j}^{\eps_{j}} \ta{\widetilde{p}_{i_j}}$; \ie we denote by $\vect{\widetilde{p}_{i_j}}$ the abelian part of the $j$-th syllable of $w[\overrightarrow{u_i \t[p_i]}]$ (after normalizing);
    \item we write $\vect{\widetilde{p}}_w =(\vect{\widetilde{\vect{p}}_{i_1}},\vect{\widetilde{p}_{i_2}},\ldots,\vect{\widetilde{\vect{p}}_{i_l}})$;
    \item we denote by $\widetilde{\matr{A}}_{w(\overrightarrow{u_i})}$ the block matrix having as $j$-th row block $\matr{A}_{w[j+1,l](\overrightarrow{u_i})}$, \ie
    \begin{equation}
    \widetilde{\matr{A}}_{w(\overrightarrow{u_i})}
    \,=\,
    \begin{pmatrix}
    \matr{A}_{w[2,l](\overrightarrow{u_i})}\\
    \matr{A}_{w[3,l](\overrightarrow{u_i})}\\
    \vdots \\
    \matr{A}_{w[l](\overrightarrow{u_i})}\\
    \matrid[m]
    \end{pmatrix}.
\end{equation}
\end{itemize}

We summarize the obtained result in the statement below.

\begin{prop} \label{prop: Ga products}
Let $\fba{u_1}{p_1},
\ldots, \fba{u_k}{p_k} \in \Ga$, and let $w$ be a nontrivial formal free word on~$k$ generators. 
Then (with the previous notation), 

\begin{equation}
   w(\overrightarrow{u_i \t[p_i]}) 
   \,=\,
   w(\overrightarrow{u_i}) \, \t^{\,\widetilde{\vect{p}}_{w}\,
\widetilde{\matr{A}}_{w(\overrightarrow{u_i})}}
\end{equation}

In particular, 
if $\fba{u}{p}\in \Ga$ and $l >0$, then
\begin{equation}
    (\fba{u}{p})^l
    \,=\,
    u^l \, \t^{\vect{p} (\ssum_{j=0}^{l-1} \matr{A}_{u}^{j})}.
\end{equation}
\end{prop}

\section{Homomorphisms 
} \label{sec: homos}
In this section, we
use the conditions given by the relators in the standard presentation \eqref{eq: presentation ABF2}
to obtain a description of the homomorphisms between FABF groups, which, as seen below, fit into two main classes.

To this end, throughout the section we assume that $m,m'\geq 1$ and $n,n'\geq 2$ are arbitrary integers, and that 
$\acti \colon \Fn \to \GLm$
and $\actii \colon \Fnn \to \GLmm$
are arbitrary homomorphisms (given by matrices $\matr{A_1}, \ldots , \matr{A_n} \in \GLm$ and 
$\matr{B_1}, \ldots, \matr{B_{n\!'}} \in \GLmm$ respectively), and we define $\Ga$ as in \eqref{eq: presentation ABF2} and
\begin{align*}
   \Gb
    &\,=\,
    \Fnn \ltimes_{\actii} \Zmm
    \,=\,
    \Pres{\hspace{-3pt}
    \begin{array}{ll}
    z_1,\ldots,z_{m\!'}\\
    y_1,\ldots,y_{n\!'}
    \end{array}
    }
    {
    \begin{array}{ll}
    z_i^{-1} z_j z_i = z_j & \forall i,j \in [1,m\!']\\
    y_i^{-1} z_j y_i = \z^{\vect{b_{ij}}} & \forall i \in [1,n\!'],\forall j\in[1,m\!']
    \end{array}
    },
\end{align*}
where $\Fnn$ is the free group with basis $(y_1,\ldots,y_{n\!'})$, and $\Zmm$ is the free-abelian group with basis $(z_1,\ldots,z_{m\!'})$.
 Multiplicative notation for abelian words is adapted in the natural way; namely, for every $\vect{b} = (b_1,\ldots,b_{m\!'}) \in \Zmm$, we write $
\z^{\vect{b}} =
z_1^{b_1} z_2^{b_2} \cdots z_{m\!'}^{b_{m\!'}}$.

In order to study homomorphisms $\Ga \to \Gb$ (between FABF groups), let us consider an arbitrary assignation of elements in $\Gb$ to the generators of $\Ga$:
\begin{equation} \label{eq: candidate 1}
    \Phi \colon
    \left\{ \!
    \begin{array}{rcll}
    t_j & \mapsto & u_j\, \z^{\vect{q_j}} = u_j \, \z^{\,\vect{e_j} \matr{Q}} , & \text{for all } j\in[1,m]\\
    x_i & \mapsto & v_i\,  \z^{\vect{p_i}} = x_i\varphi \, \z^{\,\vect{f_i} \matr{P}}, & \text{for all } i\in[1,n] \ \,
    \end{array}
    \right.
\end{equation}
where $\vect{q_j},\vect{p_i} \in \Zmm$ and $u_j,v_i \in \Free[n\!']$, for every $j \in [1,m]$ and every $i \in [1,n]$; and we have denoted by: 
\begin{itemize}
    \item $\matr{Q}$
the $m \times m'$ integer matrix having as $j$-th row the vector $\vect{q_j}$ ($j = 1, \ldots, m$);
    \item $\matr{P}$
the $n \times m'$ integer matrix having as $i$-th row the vector $\vect{p_i}$ ($i = 1, \ldots, n$);
    \item $\varphi$ the homomorphism of $\Fn \to \Free[n\!']$ given by the assignation $x_i \mapsto v_i$.
    \item $(\vect{e_j})_{j=1}^{m}$ and $(\vect{f_i})_{i=1}^{n}$ the canonical bases of $\ZZ^m$ and~$\ZZ^n$, respectively.
\end{itemize}

As is well known, for $\Phi$ to define a homomorphism $\Ga \to \Gb$ it is necessary (and sufficient) that $\Phi$ converts the relators for $\Ga$ into relations for $\Gb$. Hence, we need to impose the following two conditions to the assignation \eqref{eq: candidate 1}:
\begin{align}
    ((t_i)\Phi)^{-1} \, (t_j)\Phi \, (t_i)\Phi &\,=_{\Gb}\, (t_j)\Phi
    && \hspace{-50pt}
    \forall i,j \in [1,m]
    \label{eq: comm}
    \tag{\texttt{comm}}\\
    ((x_i)\Phi)^{-1} \, (t_j)\Phi \, (x_i)\Phi &\,=_{\Gb}\, ((t_j)\auti_i)\Phi
    && \hspace{-50pt}
    \forall j \in [1,m] \ \forall i \in [1,n]
    \label{eq: conj}
    \tag{\texttt{conj}}
    \end{align}  
If it is clear from the context, throughout the proof we will omit the $\Gb$ subscript and we will assume that equalities involving words in the generators for $\Gb$ work modulo the relations in $\Gb$.  Applying the description of $\Phi$
to the commutativity conditions \eqref{eq: comm} we obtain:
\begin{align*}
    (u_i \,\z^{\vect{q_i}})^{-1}\, u_j \, \z^{\vect{q_j}}\, u_i z^{\vect{q_i}}
    &\,=\, u_j \, \z^{\vect{q_j}}
    && \hspace{-50pt}
    \forall i,j \in [1,m],
\end{align*}
which, after normalizing, separating the free and free-abelian parts, and rearranging, takes the form: 
\begin{align}
    u_i^{-1} u_j u_i &\,=\, u_j 
    \label{eq: comm free 0}\\
    \vect{q_i} (\matr{B}_{u_i}^{-1}\matr{B}_{u_j}\matr{B}_{u_i} - \Imm)
     &\,=\, 
    \vect{q_j} (\matr{B}_{u_i} - \Imm )
    \label{eq: comm ab 0}
\end{align}
for every $i,j \in [1,m]$.

The first set of conditions \eqref{eq: comm free 0} is saying that (for $\Phi$ to be a homomorphism)  all the $u_j$'s in~\eqref{eq: candidate 1} must commute with each other.
In the context of free groups, it is easy to see that this means that all the $u_i$'s must be powers of a common element $v \in \Fnn$.
More precisely, the set of conditions \eqref{eq: comm free 0} is equivalent to the existence of some $v \in \Fnn$ (which can be assumed to be a non proper power) and integers $r_1,\ldots,r_m \in \ZZ$ such that $u_j \,=\, v^{r_j}$, for every  $j \in [1,m]$. Hence, the free and free-abelian parts of \eqref{eq: comm}
 take the simplified form:
\begin{align}
    u_j &\,=\, v^{r_j} &&\hspace{-80pt} \forall j\in[1,m] \label{eq: comm.fr} \tag{\texttt{comm.fr}}\\
    \vect{q_i}(\matr{B}_{v}^{r_j}-\Imm) 
    &\,=\,
    \vect{q_j} (\matr{B}_{v}^{r_i}- \Imm) 
    &&\hspace{-80pt} \forall i,j\in[1,m] \label{eq: comm.ab}\tag{\texttt{comm.ab}}
\end{align}
where $v \in \Fnn$ is a non proper power, and $\vect{r} = (r_1,\ldots,r_m) \in \ZZ^m$. 
That is, 
our general candidate homomorphism $\Ga \to \Gb$ takes the form:
\begin{equation} \label{eq: candidate 3}
    \Phi \colon
    \left\{ \!
    \begin{array}{rcll}
    t_j & \mapsto & v^{r_j}\, \z^{\vect{q_j}}, & \text{for all } j\in[1,m]\\
    x_i & \mapsto & v_i\, \z^{\vect{p_i}}& \text{for all } i\in[1,n] \ \,
    \end{array}
    \right.
\end{equation}
where $v,v_i \in \Fnn$ ($i \in [1,n]$), $v$ is a non-proper power, 
$\vect{r} = (r_1,\ldots,r_m) \in \ZZ^m$,
$\vect{q_j},\vect{p_i} \in \Zmm$ for every $j \in [1,m]$ and $i \in [1,n]$,  and
 $\vect{q_i}(\matr{B}_{v}^{r_j}- \Imm)
 =
 \vect{q_j} (\matr{B}_{v}^{r_i}- \Imm)$
 for every $i,j \in [1,m]$.
Then, for every $\vect{a} = (a_1,\ldots,a_m) \in \ZZ^m$,
\begin{equation} \label{eq: a Phi pi}
    (\t^{\vect{a}}) \Phi \pi
    \,=\,
    v^{\sum_{j=1}^{m} r_j a_j}
    \,=\,
     v^{\vect{r}  \vect{a}^\T} . 
\end{equation}

Now, applying the description of $\Phi$ in \eqref{eq: candidate 3} to the conjugacy conditions in \eqref{eq: conj} we obtain:
\begin{equation*} \label{eq: conj II}
\begin{aligned} 
\z^{-\vect{p_i}}\, v_i^{-1} \, v^{r_j} \, \z^{\vect{q_j}} \, v_i \, \z^{\vect{p_i}}
    &\,=\, ((t_j)\auti_i)\Phi
    & \forall j \in [1,m]\  \forall i \in [1,n]
    \\
    v_i^{-1} \, v^{r_j} \, v_i \, \z^{-\vect{p_i} \matr{B}_{v_i}^{-1} \matr{B}_{v}^{r_j} \matr{B}_{v_i} + \vect{q_j} \matr{B}_{v_i} + \vect{p_i}} 
    &\,=\, (\t^{
    \vect{a_{i,j}}})\Phi
    & \quad \forall j \in [1,m] \ \forall i \in [1,n],
\end{aligned}
\end{equation*}
where $\vect{a_{ij}} = \vect{e_j} \matr{A_i}$ denotes the $j$-th row of the matrix $\matr{A_i}$.
After, equalizing the free and abelian parts, and taking into account \eqref{eq: a Phi pi}, we obtain the conditions:
\begin{align} 
    v_i^{-1}\, v^{r_j}\, v_i &\,=\, v^{\vect{r}\, \vect{a^\T_{i,j}}}
      & \forall j \in [1,m] \ \forall i \in [1,n]\tag{\texttt{conj.fr}} \label{eq: conj.fr}\\
      \vect{q_j} \matr{B}_{v_i} +
   \vect{p_i}(\Imm- \matr{B}_{v_i}^{-1} \matr{B}_{v}^{r_j} \matr{B}_{v_i})
    &\,=\, (\t^{\vect{a_{i,j}}})\Phi\tau
    & \forall j \in [1,m] \ \forall i \in [1,n]\tag{\texttt{conj.ab}}\label{eq: conj.ab}
\end{align}

At this point, it is convenient to split our candidate homomorphisms \eqref{eq: candidate 3} into two cases: those where
$u_j = v^{r_j} = \trivial$ for every ${j \in [1,m]}$ (\defin{homomorphisms of type I}),
and those where
at least one of the 
$u_j = v^{r_j}$'s is nontrivial (\defin{homomorphisms of type II}).
As we will see, this second case corresponds to a degenerated situation; we consider it in the first place.

\textbf{Homomorphisms of type II:} At least one of the $u_j = v^{r_j}$ is nontrivial (or equivalently, $v \neq \trivial$ and $\vect{r} \neq \vect{0}$) in \eqref{eq: candidate 3}. 

We recall the well-known result below.
\begin{lem}
Let $v \in \Fn \setmin \set{\trivial}$ be a
nontrivial
non proper power. If there exist $w \in \Fn$ and $p,q \in \ZZ$ such that $w^{-1} v^p w = v^q$, then $p = q$ and $w\in \gen{v}$.
\end{lem}

Since, by hypothesis,  $v \neq \trivial$, after applying the previous lemma to the conjugacy condition \eqref{eq: conj.fr} we 
deduce  that for homomorphisms of type II:
\begin{equation*}
r_j=\vect{r}\cdot \vect{a_{i,j}^\T}\qquad \forall (i,j)\in [1,n]\times [1,m] \,,  
\end{equation*}
which can be rewritten as
\begin{equation*}
\vect{r}\matr{A_i^\T} \,=\, \vect{r}  \quad \forall i \in [1,n] \,,
\end{equation*}
or equivalently as
\begin{equation*}
    \vect{r} \in \textstyle{\bigcap_{i=1}^{n} \Fix(\matr{A_i^\T})} \,; \tag{\texttt{conj.fr.II.a}}
\end{equation*}
and there exists $\vect{s} = (s_1,\ldots,s_n) \in \ZZ^n$ such that
\begin{equation}
    v_i \,=\, v^{s_i}  \quad \forall i \in [1,n] \,. \tag{\texttt{conj.fr.II.b}}
\end{equation}
Therefore, the condition \eqref{eq: conj.ab} can be simplified to:
 \begin{equation} \label{eq: conj.ab.II}
   \vect{p_i}(\matr{B}_{v}^{r_j} - \Imm)
    \,=\, \vect{q_j} \matr{B}_{v}^{s_i} - (\t^{\vect{a_{i,j}}})\Phi\tau 
    \qquad\forall j \in [1,m] \ \forall i \in [1,n]\tag{\texttt{conj.ab.II}}
\end{equation}

\medskip

\textbf{Homomorphisms of type I:} every $u_j$ ($j \in [1,n]$) is trivial in \eqref{eq: candidate 3}; that is, the candidate homomorphisms of type I are of the form:
\begin{equation} \label{eq: candidate I}
    \Phi \colon
    \left\{ \!
    \begin{array}{rcll}
    t_j & \mapsto & \z^{\vect{q_j}}, & \text{for all } j\in[1,m]\\
    x_i & \mapsto & v_i\, \z^{\vect{p_i}}& \text{for all } i\in[1,n] \,.
    \end{array}
    \right.
\end{equation}
Note that then
for every $\vect{a} \in \ZZ^m$, $(\t^{\vect{a}})\Phi = (\t^{\vect{a}})\Phi \tau = \z^{\vect{a} \matr{Q}} \in \Zmm$. That is, the restriction of $\Phi$ to the abelian part is a homomorphism of free-abelian groups; namely, $\Phi_{|\Zm} = \matr{Q} \colon \Zm \to \Zmm$.
In particular, the conditions \eqref{eq: comm} and \eqref{eq: conj.fr} trivially hold, whereas \eqref{eq: conj.ab} can be rewritten as:
\begin{equation*}
    \vect{q_j} \matr{B}_{v_i}
    \,=\, \vect{a_{i,j}} \matr{Q}
     \quad \forall j \in [1,m] \ \forall i \in [1,n]
\end{equation*}
or, more compactly, as:
\begin{equation}\label{eq: conj.ab.I}
    \matr{Q} \matr{B}_{v_i}
    \,=\, \matr{A_i} \matr{Q}
    \quad \forall i \in [1,n] \tag{\texttt{conj.ab.I}}
\end{equation}


Putting together the conditions obtained for endomorphisms of types I and II, we reach the following description of endomorphisms between FABF groups.

\begin{thm} \label{thm: homos FABF}
Let $\Ga = \Fn \ltimes_{\acti} \Zm$ and $\Gb = \Fnn \ltimes_{\actii} \Zmm$ be FABF groups. Then, the following is  a complete list of the homomorphisms $\Ga \to \Gb$:
\begin{enumerate}[dep]
    \item Homomorphisms of type I
\begin{equation} \label{eq: type I}
    \hI{\varphi}{\matr{Q}}{\matr{P}} \colon
    \left\{ \!
    \begin{array}{rcll}
    t_j & \mapsto &\z^{\vect{q_j}}, & \text{for all } j\in[1,m]\\
    x_i & \mapsto & x_i \varphi\, \z^{\vect{p_i}}& \text{for all } i\in[1,n]
    \end{array}
    \right.\,,
\end{equation}
where 
$\varphi \in \Hom(\Fn,\Fnn)$, $\matr{Q} \in \mathcal{M}_{m \times m\!'}(\ZZ)$ and $\matr{P} \in \mathcal{M}_{n \times m\!'}(\ZZ)$ satisfying
\begin{equation}\label{eq: cond_I}
    \matr{Q} \matr{B}_{x_i\! \varphi}
    \,=\, \matr{A_i} \matr{Q}
    \quad \forall i \in [1,n] \,.
\end{equation}

\item Homomorphisms of type II
    \begin{equation} \label{eq: type II}
    \Phi_{v,\vect{r},\vect{s},\matr{Q},\matr{P}} \colon
    \left\{ \!
    \begin{array}{rcll}
    t_j & \mapsto & v^{r_j}\, \z^{\vect{q_j}}, & \text{for all } j\in[1,m]\\
    x_i & \mapsto & v^{s_i}\, \z^{\vect{p_i}}& \text{for all } i\in[1,n] 
    \end{array}
    \right. \, ,
\end{equation}
where $v \in \Fnn \setmin\set{\trivial}$ is a non-proper power,
$\vect{s} \in \ZZ^n$,
$\vect{r} \in \bigcap_{i=1}^{n} \Fix(\matr{A_i^\T}) \setmin \set{\vect{0}}$,
$\matr{Q} \in M_{m\times m'}(\ZZ)$, $\matr{P} \in M_{n\times m'}(\ZZ)$,  and
 \begin{align}[left=\empheqlbrace]\,
     \vect{q_j} (\matr{B}_{v}^{r_i}- \Imm)
     &\,=\, \vect{q_i} (\matr{B}_{v}^{r_j}- \Imm)
     &&\hspace{-50pt} \forall i,j \in [1,m] \\
    \vect{p_i}(\matr{B}_{v}^{r_j} - \Imm)
    &\,=\, \vect{q_j} \matr{B}_{v}^{s_i} - (\t^{\vect{a_{i,j}}})\Phi\tau 
    &&\hspace{-50pt}\forall j \in [1,m] \ \forall i \in [1,n]
\end{align}
\end{enumerate}
\end{thm}
We denote by $\Hom_I(\Ga,\Gb)$ (\resp $\Hom_{I\!I}(\Ga,\Gb)$) the \defin{set of homomorphisms of type I} (\resp II) from $\Ga$ to $\Gb$.

Taking advantage of the formulas obtained in \Cref{prop: Ga products} for arbitrary products in~$\Ga$, we can give general expressions for the images of homomorphisms between FABF groups, and derive expressions for composition and powers of endomorphisms of type I.



\begin{cor} \label{cor: homos FABF I}
Let $\Phi = 
\Phi_{\varphi,\matr{Q},\matr{P}} \in \Hom_I(\Ga,\Gb)$. Then, for every $\fba{u}{a} \in \Ga$,

\begin{align}
    (\fba{u}{a}) \Phi
    &\,=\,
    u \varphi \, \z^{\, \vect{a} \matr{Q} \,+\, u\bbeta_{\!\Phi}},
\end{align}
That is,
\begin{align}
    (u \t[a]) \Phi \pi 
    &\,=\, (u \t[a]) \pi\, \Phi 
    = u \varphi, \text{  and} \label{eq: EndI pi}\\
    (u \t[a]) \Phi \, \tau
    &\,=\, (u \t[a]) \tau\, \matr{Q} + (u \t[a]) \pi \, \bbeta_{\!\Phi}
    \,=\, \vect{a} \matr{Q} + u\bbeta_{\!\Phi}. \label{eq: EndI tau}
\end{align}
where $\bbeta_{\!\Phi}\colon \Fn \to \Zmm$ such that
$(u)\bbeta_{\!\Phi} = \widetilde{\vect{p}}_u  \widetilde{\matr{B}}_{u\!\varphi}$ if $u\neq \trivial$, and $(\trivial)\bbeta_{\Phi} = \vect{0}$.

\end{cor}

Hence
$\Phi_{|\ZZ^m} = \matr{Q}$ (\ie  for every $\vect{a} \in \Zm$, $(\t^{\vect{a}}) \Phi = \z^{\vect{a} \matr{Q}}$)
and $\Phi_{\mid \Fn} \pi = \varphi$. In particular, every homomorphism \smash{$\matr{Q} \colon \ZZ^m \to \ZZ^{m'}$}
(\resp $\varphi \colon \Free[n] \to \Free[n']$) can be seen as a restriction of
$\Phi_{\matr{Q},\varphi,\matr{0}} \in \Hom_I(\Ga,\Gb)$.



\begin{cor}
Let 
$\Phi=\Phi_{\varphi,\matr{Q},\matr{P}}$ and $\Phi'=\Phi_{\varphi',\matr{Q'},\matr{P'}}$ be type I endomorphisms of $\Ga$.
Then, for every $\fba{u}{a} \in \Ga$,
\begin{align}
(\fba{u}{a}) \Phi \Phi'
    \,=\, 
    (\fba{u}{a}) \Phi_{\varphi,\matr{Q},\matr{P}} \,
    \Phi_{\varphi',\matr{Q'},\matr{P'}}
    \,=\, 
    u\varphi \varphi' \
    \t^{\vect{a} \matr{Q} \matr{Q'} + u \bbeta_{\!\Phi} \matr{Q'} + u\!\varphi \bbeta_{\!\Phi\!'}}
\end{align}
\end{cor}

The expression for powers of endomorphisms of type I  follows easily after applying \Cref{eq: EndI pi,eq: EndI tau} recursively.

\begin{cor} \label{cor: EndI^k}
Let $\Phi = \Phi_{\varphi,\matr{Q},\matr{P}} \in \End_{I}(\Ga)$, and let $k>0$. Then, for every $\fba{u}{a} \in \Ga$,
\begin{align} \label{eq: EndI^k}
    (\fba{u}{a}) \Phi^k
    &\,=\, u \varphi^{\!k} \
    \t^{\,\vect{a} \matr{Q}^{k}
    \, + \,
    \sum_{i=0}^{k-1} u \varphi^{\!i} \bbeta_{\!\Phi} \,\matr{Q}^{k-i-1}}
\end{align}


\end{cor}

\begin{proof}
It is clear from \Cref{eq: EndI pi} that 
$(u \t[a]) \Phi^{k} \pi = ((u \t[a]) \Phi^{k-1}) \Phi \pi = ((u \t[a]) \Phi^{k-1}) \pi\varphi$. Since $(u \t[a]) \Phi^{0} \pi = u $, we have that
\begin{equation}
(u \t[a]) \Phi^k \pi \,=\,  u \varphi^k .
\end{equation}
Similarly, from \Cref{eq: EndI tau}, and using the last result:
\begin{align}
    (u \t[a]) \Phi^{k} \tau &\,=\, ((u \t[a]) \Phi^{k-1}) \Phi \tau \notag \\
    &\,=\, (u \t[a]) \Phi^{k-1} \tau \matr{Q} + (u \t[a]) \Phi^{k-1} \pi \bbeta_{\!\Phi} \notag\\
    &\,=\, (u \t[a]) \Phi^{k-1} \tau \matr{Q} + u \varphi^{k-1} \bbeta_{\!\Phi}. \label{eq: EndI^k tau rec}
\end{align}
Again, taking into account that $(u \t[a]) \Phi^{0} \tau  = \vect{a}$, we have that
\begin{equation}
(u \t[a]) \Phi^k \tau
     \,=\, \vect{a} \matr{Q}^{k}
    \, + \,
    \textstyle{
    \sum_{i=0}^{k-1} u \varphi^{\!i} \bbeta_{\!\Phi} \,\matr{Q}^{k-i-1}}
\end{equation}
and \eqref{eq: EndI^k} follows.
\end{proof}

\begin{rem} \label{rem: hom II}
Let $ \Phi
= \Phi_{v,\vect{r},\vect{s},\matr{Q},\matr{P}} \in \Hom_{I\!I}(\Ga,\Gb)$. Then, for every $\fba{u}{a} \in \Ga$,
$(\fba{u}{a})\Phi \pi = v^{\vect{a} \vect{r} + \vect{u} \vect{s}}$,
where $\vect{u} \in \ZZ^n$ denotes the abelianization of $u \in \Fn$.
In particular,  $(\Ga) \Phi \pi$ is a cyclic subgroup of $\Fnn$.
\end{rem}

\begin{lem} \label{lem: typeII not monos epis}
Type II homomorphisms are neither injective nor surjective.
\end{lem}

\begin{proof}
Note that if $\Phi$ is of type II, namely $\Phi = \Phi_{v,\vect{r},\vect{s},\matr{Q},\matr{P}}$, then the image of any commutator in~$\Fn$ is abelian;
more precisely,
for every $u,w \in \Fn$, 
\begin{equation}
    (\comm{u}{w}) \Phi \pi
    \,=\,
    (u^{-1} w^{-1} u w) \Phi \pi
    \,=\,
    v ^{-\vect{u} \vect{s}} v^{- \vect{w} \vect{s}}  v^{ \vect{u} \vect{s}}  v^{\vect{w} \vect{s}}
    \,=\,
    \trivial,
\end{equation}
and hence $([u,w])\Phi \in \Zmm$. But, since $\Zmm$ is abelian, this means that the image by $\Phi$ of any commutator of the form $[[u,w][u',w']]$ (\ie any commutator between commutators in $\Fn$) must be trivial, and hence $\Phi$ is necessarily not injective.

Non-surjectivity is obvious from
$((\Ga) \Phi_{u,\matr{Q},\matr{P}}) \pi = \gen{u^{\gcd(\vect{r},\vect{s})}} \neq \Fnn$,
since we are assuming~$n'>1$.
\end{proof}

\begin{prop} \label{prop: injective hom}
Let $\Phi\colon \Ga \to \Gb$ be a homomorphism. Then $\Phi$ is injective if and only if it is of type I,
$\Phi = \Phi_{\varphi,\matr{Q},\matr{P}}$,
with both $\varphi$ and  $\matr{Q}$ injective.
\end{prop}

\begin{proof}
$[\Rightarrow]$ From \Cref{lem: typeII not monos epis} it is immediate that injective homomorphisms must be of type~I, \ie of the form $\Phi = \Phi_{\varphi,\matr{Q},\matr{P}}$. Moreover, 
if $\Phi$ is injective, then its restriction to the abelian part, an hence $\matr{Q}$, must be injective as well.
Now suppose that $\varphi$ is not injective; that is, $\Trivial \neq \ker \varphi \normaleq \Fnn$. Since $n'\geq 2$ (by assumption), $\ker \varphi$ is not cyclic and hence there exist non-commuting elements $w,v \in \ker \varphi$. But
it is easy to see that $(\ker \varphi)\Phi \subseteq \Zmm$, which is abelian. Therefore $([w,v])\Phi = \trivial$
and $\Phi$ is not injective, in contradiction with the hypothesis. Therefore both $\matr{Q}$ and $\varphi$ must be injective, as claimed.

$[\Leftarrow]$ Conversely, suppose that $\Phi = \Phi_{\varphi,\matr{Q},\matr{P}}$ with both $\varphi$ and $\matr{Q}$ injective, and recall that
$(u \t^{\vect{a}}) \Phi = 
u \varphi \, \z^{\, \vect{a} \matr{Q} \,+\, u\bbeta_{\!\Phi}}$.
Now, if $(u \t^{\vect{a}})\Phi = \trivial$ then both the projection to the free and to the free-abelian part must be trivial. That is, on one side  $(u \t^{\vect{a}}) \Phi \pi = u\varphi = \trivial$ and hence $u = \trivial$
(since $\varphi$ is injective); and on the other side, using that $u = \trivial$, we have that
$(u \t^{\vect{a}}) \Phi \tau =
\vect{a} \matr{Q} \,+\, u  \bbeta_{\!\Phi}
= \vect{a} \matr{Q} = \vect{0}$, and hence $\vect{a} = \vect{0}$ (since $\matr{Q}$ is injective). Therefore $\Phi$ is injective and the proof is complete.
\end{proof}

\begin{prop} \label{prop: bijective hom}
Let $\Phi\colon \Ga \to \Gb$ be a homomorphism. Then $\Phi$ is bijective if and only if it is of type I,
$\Phi = \Phi_{\varphi,\matr{Q},\matr{P}}$,
with both $\varphi$ and  $\matr{Q}$ bijective.
\end{prop}

\begin{proof}
$[\Rightarrow]$
If $\Phi$ is bijective then it is clear (from \Cref{prop: injective hom}) that both $\varphi$ and $\matr{Q}$ must be injective. On the other hand, since $\Phi$ is surjective
$\Fnn = (\Ga) \Phi \pi = (\Fn) \varphi $, and therefore $\varphi$ is surjective, and hence bijective. 
Then, the image under $\Phi$ of any non-abelian element $u \t^{\vect{a}} \in \Ga \setmin \ZZ^m$ is necessarily non-abelian. Therefore, $\im \Phi \cap \Zmm = \im \Phi_{| \Zm} \cap \Zmm = \im \Phi_{|\Zm} = \im \matr{Q}$. Since $\Phi$ is surjective, then $\im \matr{Q} = \Zmm$ and therefore $\matr{Q}$ is also surjective, and hence bijective.

$[\Leftarrow]$ From \Cref{prop: injective hom} it is enough to see that if $\varphi$ and $\matr{Q}$ are bijective, then $\Phi = \Phi_{\varphi,\matr{Q},\matr{P}}$ is surjective. But this is clear since for every $v \, \z^{\vect{b}} \in \Gb$, putting $w=v\varphi^{-1}$ and $\vect{c}=(\vect{b}-w\Phi\tau)\matr{Q}^{-1}$, 
we have that $(w\ta{\vect{c}})\Phi=v\z^{\vect{b}}.$
\end{proof}

The following is an immediate consequence of \Cref{prop: injective hom}, \Cref{prop: bijective hom} and the fact that free groups and free-abelian groups are Hopfian.
\begin{prop} \label{prop: autos}
Let $\Phi$ be an endomorphism of $\Ga$. Then, the following statements are equivalent: 
\begin{enumerate}[ind]
    \item $\Phi$ is bijective (\ie $\Phi$ is an automorphism of $\Ga$)
    \item $\Phi$ is onto (\ie $\Phi$ is an epimorphism of $\Ga$);
    \item \label{item: auto description} $\Phi$ is of type I, $\Phi = \Phi_{\varphi,\matr{Q},\matr{P}}$, with with both $\varphi$ and $\matr{Q}$ surjective (and hence bijective). 
\end{enumerate}
\end{prop}

\begin{cor}
FABF groups $\Ga = \ABFa$ are Hopfian and not coHopfian.
\end{cor}

\begin{proof}
The Hopfianity
of $\Ga$ is immediate from \Cref{prop: bijective hom}, whereas non-coHopfianity follows from \Cref{prop: injective hom,prop: bijective hom} and the non-coHopfianity of $\ZZ^m$ (or $\Fn$).
\end{proof}

\section{The Isomorphism Problem} \label{sec: IP}


Let $\mathcal{P}$ be a family of finite presentations, and let us denote by $\gen{P}$ the group presented by a presentation $P = (X\mid R)$; that is, $\gen{P} = \pres{X}{R} = \Free[X]/\normalcl{R}$.

\begin{named}[(Group) Isomorphism problem within $\mathcal{P}$, $\IP(\mathcal{P})$]
Decide, given $P_1,P_2 \in \mathcal{P}$, whether $\gen{P_1} \isom \gen{P_2}$.
\end{named}

Accordingly, if $\mathcal{G}$ is a family of finitely presented groups, the \defin{isomorphism problem for~$\mathcal{G}$}, denoted by $\IP(\mathcal{G})$, consists of deciding, given two arbitrary finite presentations of groups in~$\mathcal{G}$, whether they present isomorphic groups.
\begin{equation}
\mathsf{\IP(\mathcal{G})} \ \equiv\ 
\gen{P_1} \isom \gen{P_2}\?_{P_1,P_2 \text{ f.p.~of groups in } \mathcal{G}}
\end{equation}

A standard argument using Tietze transformations justifies that, in fact,
it is enough to restrict the inputs of $\IP(\mathcal{G})$  to any recursively enumerable family $\mathcal{P}$ of finite presentations covering $\mathcal{G}$ (\ie such that $\set{\gen{P} \st P \in \mathcal{P}} = \mathcal{G}$).

\begin{lem}\label{lem: tietze transformations}
Let $\mathcal{G}$ be a family of finitely presented groups, and let $\mathcal{P}$ be a recursively enumerable family of finite presentations covering $\mathcal{G}$. Then, given a finite presentation $Q$ of a group in $\mathcal{G}$, a presentation $P \in \mathcal{P}$ such that $\gen{P} \isom \gen{Q}$ is computable.
\end{lem}

\begin{proof}
Given a finite presentation $Q$ of a group in $\mathcal{G}$, we can start a diagonal process successively applying to $Q$ all the possible sequences of Tietze transformations, thus generating an enumeration $(Q_n)_{n\geq 0}$ of all the finite presentations of $\gen{Q}$. Since this process will potentially reach every finite presentation for $\gen{Q}$, and $\mathcal{P}$ covers $\mathcal{G} \ni \gen{Q}$, at some point it will necessarily reach a presentation $Q_n \in \mathcal{P}$ (such that $\gen{Q_n}\isom \gen{Q}$). Finally, we can algorithmically detect this match since, as $\mathcal{P}$ is recursively enumerable, for every presentation $Q_n$ obtained, we can produce an enumeration $(P_m)_{m \geq 0}$ of $\mathcal{P}$, and keep checking whether $P_m = Q_n$ until a guaranteed match happens.
\end{proof}

A fundamental result in algorithmic group theory (answering the third of Dehn's seminal problems in \cite{dehn_uber_1911}) is that the most general version of \IP\ (corresponding to $\mathcal{P}$ being the whole family of finite presentations) is algorithmically undecidable: there is no algorithm to decide whether two given finite presentations present isomorphic groups. Subsequently, restrictions of \IP\ to many different families of groups have been considered with both positive (\eg hyperbolic groups) and negative (\eg solvable groups) results.

In the unpublished paper \cite{levitt_unsolvability_2008}, \citeauthor{levitt_unsolvability_2008}, inspired by the ideas that led to the undecidability of the conjugacy problem in \cite{bogopolski_orbit_2010}, uses a brief and smart argument to prove the undecidability of the isomorphism problem for free-abelian-by-free groups.

\begin{thm}[\citeny{levitt_unsolvability_2008}]
The isomorphism problem $\IP(\FABF)$ is undecidable. \qed
\end{thm}

The purpose of this section is to investigate possible \IP-decidable subfamilies of \FABF\ groups.
Natural candidates are \FABF\ groups given by actions $\acti\colon \Fn \to \GL_m(\ZZ)$ with finite generated kernel; namely, when the $\acti$ is either injective (\aka \defin{faithful actions}) or has finite image (\aka \emph{finite actions}). In the case of finite actions, we are able to prove decidability of the $\IP$, and in the case of faithful actions, we see that the $\IP$ is Turing-equivalent to some constrained subgroup conjugacy problem in $\GLm$. We ignore whether this last problem is decidable or not.

Since (by definition) the family of presentations \eqref{eq: presentation ABF} (or \eqref{eq: presentation ABF2}) covers the family of FABF groups, we can assume that the inputs of $\IP(\FABF)$ are in this form; or equivalently, that the inputs are $n, n'\geq2$, $m,m'\geq 1$, and matrices $\matr{A_1}, \ldots \matr{A_n} \in \GL_m(\ZZ)$ and $\matr{B_1}, \ldots,\matr{B_{n\!'}} \in \GL_{m\!'}(\ZZ)$.
That is,
\begin{equation}\label{eq: IP FABF inputs}
\mathsf{\IP(\FABF)} \ \equiv\ 
\Ga \isom \Gb\?_{\subalign{
&n\geq 2 \,,\, m\geq 1 \,,\, \matr{A_1}, \ldots \matr{A_n} \in \,\GL_m(\ZZ)\\
&n'\geq 2 \,,\, m'\geq 1 \,,\, \matr{B_1}, \ldots \matr{B_{n'}} \in \,\GL_{m'}(\ZZ)
}}
\end{equation}


From \Cref{thm: homos FABF} and \Cref{prop: bijective hom} it follows that the isomorphisms $\Phi \colon \Ga \to \Gb$ are of the form $\Phi_{\varphi,\matr{Q},\matr{P}}$ (\ie of type I), where both $\varphi\colon \Fn \to \Fnn$ and $\matr{Q}\colon \Zm \to \Zmm$ are isomorphisms. In particular this implies that $n=n'$ and $m=m'$. 

Moreover, taking advantage of the invertibility of $\matr{Q}$,
the conditions in \eqref{eq: conj.ab.I} can be rewritten for isomorphisms as
\begin{equation}\label{eq: conj ab aut}
\begin{aligned}
    \matr{B}_{x_i\!\varphi} &\,=\, \matr{Q}^{-1} \matr{A_i} \, \matr{Q} \quad \text{for all } i\in[1,n] 
\end{aligned} \tag{\texttt{conj.ab.aut}}
\end{equation} 
This kind of condition appears in the literature under the name of \emph{multiple conjugacy}. We state below its formal definition and the natural decision problem involving it.

\begin{defn}
    Let $G$ be a group, let $n\geq 2$ and let $\overline{g}=(g_1,\ldots,g_n),\overline{h}= (h_1,\ldots , h_n) \in G^n$ two $n$-tuples of elements in $G$. 
    We say that $\overline{g}$ and $\overline{h}$ are multiply conjugate, and we write $\overline{g} \rconj \overline{h}$, if there exists some (common) $x \in G$, such that $x^{-1} g_i x = h_i$ for each $i \in [1,n]$. We write $\overline{g} \conj_x \overline{h}$ if we want to emphasize the \defin{multiple conjugate} $x$.
\end{defn}

\begin{named}[Multiple conjugacy problem for $G$, $\MCP(G)$]
Decide, given $n\geq 1$, and $\overline{g}=(g_1,\ldots,g_n),\overline{h}= (h_1,\ldots , h_n) \in G^n$, whether $\overline{g} \rconj \overline{h}$.
\end{named}

With the abuse of language  $\acti = (\matr{A_1},\ldots,\matr{A_n}), \varphi\actii=(\matr{B}_{x_1 \!\varphi},\ldots,\matr{B}_{x_n \!\varphi})$ (identifying the actions with the tuples of matrices defining them) the previous considerations lead to a concise description of the condition for two \FABF\ groups to be isomorphic. 


\begin{thm} \label{thm: isom FABF}
The groups $\Ga = \ABFa$ and $\Gb = \Free[n'] \ltimes_{\actii} \ZZ^{m'}$ are isomorphic
if and only if $n=n'$, $m=m'$, and there exists $\varphi\in \Aut(\Fn)$ such that $\varphi \actii$ and  $\acti$ are rigidly conjugate. 
Briefly,
\begin{empheq}[left={
\ABFa \,\isom\, \ABFbb
\ \Leftrightarrow \ \empheqlbrace}]{align}
      & \, n = n',\ m=m' \label{eq: IP cond1}\\
      & \, \exists \varphi\in \Aut(\Fn)\, \st \varphi \actii \rconj \acti  \label{eq: IP cond2}
\end{empheq}
\end{thm}

That is, only FABF groups with the same factor ranks can be isomorphic, and they are so if and only if the condition \eqref{eq: IP cond2} holds for the corresponding ranks. 
Since the ranks $n,m,n',m'$ are computable from any finite presentation, the condition \eqref{eq: IP cond1} is immediately checkable
and we can reduce $\IP(\FABF)$ to the isomorphism problem for FABF groups with he same factor ranks, and hence to condition \eqref{eq: IP cond2}.


\begin{cor} \label{cor: und cond}
    There exist integers $m\geq 1$ and $n \geq 2$ such that, given $\acti = (\matr{A_1},\ldots,\matr{A_n})$ and
    $\actii = (\matr{B_1},\ldots,\matr{B_n})$ $n$-tuples of matrices in $\GL_m(\ZZ)$, it is algorithmically undecidable whether there exists $\varphi \in \Aut(\Fn)$ such that $\varphi\actii \rconj \acti$; or equivalently, whether there exists 
 $\varphi \in \Aut(\Fn)$ and $\matr{Q} \in \GL_m(\ZZ)$ such that $\varphi \actii = \acti \inn[\matr{Q}]$ (see the diagram below).
 \begin{figure}[H]
\begin{equation} \label{eq: FABF IP condition}
\begin{tikzcd}[ampersand replacement=\&]
\Fn \arrow{r}{\acti} \arrow[d,dashed,"\rotatebox{-90}{$\isom$}","\varphi"']\& \GL_m(\ZZ) \arrow[d,dashed,"{\inn[\matr{Q}]}"] \\
\Fn \arrow{r}{\actii} \& \GL_m(\ZZ)
\end{tikzcd}
\end{equation}
 \caption{(Undecidable) isomorphism condition for \FABF\ groups}
 \end{figure}
\end{cor}

Conditions of this kind have already appeared in the literature associated to isomorphisms between specific types of extensions of free groups (see \cite{bogopolski_automorphism_2007,cavallo_algorithmic_2017,ghys_stabilite_1980}). In fact, it is not difficult to see that 
a sufficient condition for two general semidirect products
of the form
${\Fn \ltimes_{\!\acti} G}$ and  ${\Fn \ltimes_{\!\actii} G}$ to be isomorphic is that 
there exists an automorphism $\varphi \in \Aut(\Fn)$ such that~$\acti$ and $\varphi \actii$ are rigidly conjugate in $\Out(G)$. 

\begin{prop} \label{prop: Ga isom Gb suff}
    Let $\Free
     = \pres{X}{-}
    $ be a free group,
    let $G
     = \pres{T}{R}
    $ be an arbitrary group,
    and consider
    the semidirect products
    $
    \Free \ltimes_{\!\acti} G$ and $
    \Free \ltimes_{\!\actii} G$.
    Then, for every triplet 
    $(\chi,\varphi,\theta) \in \Aut(G) \times \Aut(\Free) \times \Hom(\Free,G)$
    such that
    $\chi^{-1} \auti_i \chi = \autii_{x_i\!\varphi}  \gamma_{\!x_i\!\theta}$
    the map
    \begin{equation} \label{eq: Ga isom Gb suff}
    \begin{array}{rcl}
    \Omega \colon \Free \ltimes_{\!\acti} G & \to & \Free \ltimes_{\!\actii} G\\
    T \ni t_j & \mapsto & (t_j)\chi\\
    X \ni x_i & \mapsto & (x_i)\varphi \, (x_i)\theta
    \end{array}
    \end{equation}
    extends to an isomorphism $\Free \ltimes_{\!\acti} G \isom \Free \ltimes_{\!\actii} G$.
    That is, if
\begin{equation} \label{eq: isom suff}
    \exists \varphi \in \Aut(\Fn) \st \acti \rconj \varphi \actii \text{ in }\Out(G)
\end{equation}
then $\Free \ltimes_{\!\acti} G \isom \Free \ltimes_{\!\actii} G$.
\end{prop}

\begin{proof}
    To see that $\Omega$ extends to a homomorphism it is enough to check that all the relators in $\Free \ltimes_{\!\acti} G  =
    \pres{T,X}{R \,, x_i^{-1} t_j x_i = (t_j)\acti_i
    \ \forall x_i \in X}$ are mapped to the trivial element in $\Free \ltimes_{\!\actii} G$. This is clear for all the relators in $R$ (since $\chi$ is an automorphism); whereas for the conjugation relators we have, for every $x_i \in X$
    \[
    \begin{array}{rcl}
        x_i^{-1} \, t_j \, x_i \, ((t_j)\acti_i)^{-1} & \xmapsto{\, \Omega\, }
        & (x_i^{-1})\Omega \ (t_j)\Omega \ (x_i)\Omega \ (((t_j)\acti_i)^{-1})\Omega 
        \,=\, \\
        && ((x_i)\Omega)^{-1} \ (t_j)\Omega \ (x_i)\Omega \ (((t_j)\acti_i)\Omega)^{-1} 
        \,=\, \\
        && ((x_i)\theta)^{-1} \ ((x_i)\varphi)^{-1} \ (t_j)\chi \ (x_i)\varphi \ (x_i)\theta \ ((t_j)\acti_i \chi)^{-1} 
        \,=\, \\
        && ((x_i)\theta)^{-1} \ 
         (t_j)\chi \actii_{x_i\!\varphi} \  
         (x_i)\theta \ 
         ((t_j)\acti_i \chi)^{-1} 
        \,=\, \\
        && (t_j)\chi \, \actii_{x_i\!\varphi} \,  
         \gamma_{\!x_i\!\theta} \ 
         ((t_j)\acti_i \chi)^{-1} 
        \,=\, \trivial \,,
    \end{array}
    \]
    where we have used the conjugacy relation in $\Free \ltimes_{\!\actii} G$ in the third equality, and our hypothesis in the last one.
    The surjectivity of the extended homomorphism $\Omega$ is clear since $(G)\Omega = (G) \chi = G$ and, for every $x_i \in X$, 
    $((x_i\varphi^{-1})(x_i^{-1}\varphi^{-1}\theta\chi^{-1}))\Omega=x_i$.

    Finally, if $(u g) \Omega = u\varphi \, u\theta \, g\chi = \trivial_{\Gb}$, then $u = u\varphi = \trivial_{\Free}$, $g = g \chi = \trivial_{G}$
    and the injectivity of $\Omega$ follows.
\end{proof}
In
\parencite{bogopolski_automorphism_2007,cavallo_algorithmic_2017,ghys_stabilite_1980} situations are considered where the condition \eqref{eq: isom suff} is also necessary (and hence characterizes isomorphy $\Free \ltimes_{\!\acti} G \isom \Free \ltimes_{\!\actii} G$).
The paper \cite{bogopolski_automorphism_2007} deals with free-by-cyclic groups $\ZZ \ltimes \Fn$.
It is shown that condition \eqref{eq: isom suff} is not necessary in general (even for free-by-cyclic groups, as there is a counterexample for $\ZZ \ltimes \Free[3]$) but is true restricted to groups of the form $\ZZ \ltimes \Free[2]$. Note that, in this last case the orbit multiple-conjugacy condition \eqref{eq: isom suff} reduces to finitely many (indeed two) instances of the (simple) conjugacy problem in $\Out(\Free[2]) \isom \GL_2(\ZZ)$, and hence $\IP(\ZZ \ltimes \Free[2])$ is decidable. 


\Cref{thm: isom FABF} shows that FABF abelian-by-free groups is another family where 
condition \eqref{eq: isom suff} (essentially) characterizes isomorphy (note that in this case $\Out(G) = \Aut(G)$ since $G$ is abelian).

\begin{lem}
\label{lem: necessary isom}
If $\Ga \isom \Gb$ then 
\begin{enumerate}[dep]
    \item $(\im \acti) \inn[\matr{Q}] = \im \actii$; hence, $\im(\acti)$ and $\im(\actii)$ are conjugate subgroups in $\GL_m(\ZZ)$;
    \item $(\ker \acti) \varphi = \ker \actii$; hence, $\ker (\acti)$ and $\ker (\actii)$ are isomorphic \emph{by an automorphism of $\Fn$}.
\end{enumerate}
\end{lem}

\begin{proof}
It is enough to take the image (\resp the kernel) on both sides of the equality $\varphi \actii = \acti \inn[\matr{Q}]$. Namely, if $\varphi \actii = \acti \inn[\matr{Q}]$ then:
\begin{enumerate*}[dep]
    \item $\im \actii = \im (\varphi \actii) = \im (\acti \inn[\matr{Q}]) = (\im \acti) \inn[\matr{Q}]$, and
    \item $(\ker \actii) \varphi^{-1} = \ker(\varphi \actii) = \ker(\acti \inn[\matr{Q}]) = \ker \acti$.
\end{enumerate*}
\end{proof}

It is also interesting to consider the restrictions of the (undecidable) condition \eqref{eq: FABF IP condition} fixing each one of the variables.

\begin{thm}[\citenr{grunewald_solubility_1979} and \citenr{sarkisjan_conjugacy_1979}]
  The multiple conjugacy problem 
$\MCP(\GL_m(\ZZ))$ is algorithmically decidable. \qed
\end{thm}

\begin{cor} \label{cor: MCP}
Given $\varphi \in \Aut(\Fn)$, it is algorithmically decidable whether there exists some $\matr{Q} \in \GL_m(\ZZ)$ such that $\varphi \actii = \acti \inn[\matr{Q}]$. 

\end{cor}


The restriction obtained by fixing the matrix $\matr{Q}$ seems to be more complicated.

\begin{qst}
    Is it decidable, given $\acti=(\matr{A_1},\ldots,\matr{A_n})$ and $\actii = (\matr{B_1},\ldots, \matr{B_n})$ whether there exists some $\varphi\in \Aut(\Fn)$ such that $\varphi\actii=\acti$?
\end{qst}

\subsection{Finite Actions}
Let $\Fn$ be the free group with basis $X=\set{x_1,\ldots,x_n}$,
let $\matr{A_1}, \ldots, \matr{A_n} \in \GL_m(\ZZ)$, and consider the homomorphism $\acti\colon \Fn \to \GL_m(\ZZ)$ given by $x_i \mapsto \matr{A_i}$ for $i \in [1,n]$. Then we write $\acti = (\matr{A_1},\ldots,\matr{A_n})$. We will now focus in the case where $\im{\acti}$ is finite (we say that the action $\acti$ is finite) and prove that the isomorphism problem is decidable in this case. 


\begin{lem} \label{lem: finite action}
If $\matr{A_1}, \ldots, \matr{A_n} \in \GL_m(\ZZ)$ are given, and the image $\im (\acti) = \gen{\matr{A_1}, \ldots, \matr{A_n}}$ is 
finite, then
\begin{enumerate}[ind]
    \item \label{item: im comp}
    the image $\im(\acti)$ is explicitly computable; more precisely, there exists an algorithm that, given $\matr{A_1}, \ldots, \matr{A_n}$ as input, outputs words $u_1,\ldots, u_p \in \Fn$ such that $\im (\acti) = \set{\matr{A}_{u_1}, \ldots, \matr{A}_{u_p}}$.
    \item \label{item: ker comp}
    the kernel $\ker(\acti)$ is a normal subgroup of finite index of $\Fn$, and hence finitely generated and computable; more precisely, there exists an algorithm that, given $\matr{A_1}, \ldots, \matr{A_n}$ as input, outputs a (finite) basis $\set{v_1,\ldots, v_q}$ for $\ker(\acti)$.
\end{enumerate}
\end{lem}

\begin{proof}
\ref{item: im comp}
In order to compute all the elements in $\im (\acti) = \gen{\matr{A_1}, \ldots, \matr{A_n}}$, 
let $S_0 = \Trivial \subseteq \Fn$,
and let $S_{k} = S_{k-1} \cdot X$ for $k \geq 1$ until $\bigcup_{i=1}^{k+1} \matr{A}_{S_i} = \bigcup_{i=1}^{k} \matr{A}_{S_i}$, which is guaranteed since $\im(\acti)$ is, by hypothesis, finite. Then $\im(\acti) = \bigcup_{i=1}^{k} \matr{A}_{S_i}$.

\ref{item: ker comp}
From the first isomorphism theorem, it is immediate that $\ker(\acti)$ is normal and of finite index in $\Fn.$ To compute a basis for $\ker(\acti)$, note that the previous procedure allows us to build the Cayley digraph of $\im(\acti)$ with respect to the generating set $\matr{A_1}, 
\ldots,\matr{A_n}$, 
which is isomorphic to the Schreier digraph of $\ker(\acti)$ \wrt $X$, which, in turn (since $\ker(\acti)$ has finite index in $\Fn$), is the Stallings automaton of $\ker(\acti)$ \wrt $X$, from which a (finite) basis for $\ker(\acti)$ is computable using standard arguments.
\end{proof}

\begin{thm} \label{thm: IP finite}
The isomorphism problem for FABF groups with finite action is decidable.
\end{thm}

\begin{proof}
By \Cref{lem: tietze transformations}, we may assume that we are given a presentation of the form \eqref{eq: presentation ABF}, from which the corresponding $m,n$ and $\acti$ are easily deducible.

Hence, we can assume that integers $n,m, n', m' \geq2$ and matrices $\matr{A_1}, \ldots \matr{A_n} \in \GL_m(\ZZ)$ and $\matr{B_1}, \ldots,\matr{B_{n\!'}} \in \GL_{m\!'}(\ZZ)$ defining finite actions $\acti$ and $\actii$ respectively (\ie such that $\im \acti = \gen{\matr{A_1}, \ldots \matr{A_n}}$ and $\im \actii = \gen{\matr{B_1}, \ldots,\matr{B_{n\!'}}}$ are finite) are given to us. Below we present an algorithm which on these data decides whether $\Ga = \Fn \ltimes_{\acti} \ZZ^m \isom \Free[n\!'] \ltimes_{\actii} \ZZ^{m\!'} = \Gb$.

First of all, according to \Cref{thm: isom FABF}, if $m \neq m'$ or $n \neq n'$, then $\Ga \nisom \Gb$, and the answer is \nop. So, we can assume $m = m'$ and $n = n'$, and so $\Ga \isom \Gb$ if and only if the condition in \Cref{eq: IP cond2} holds, i.e., if:
\begin{equation*} 
\text{there exist $\varphi \in \Aut(\Fn)$ and $\matr{Q} \in \GL_m(\ZZ)$ such that $\varphi \actii = \acti \inn[\matr{Q}]$.}
\end{equation*}


On the other side, \Cref{lem: necessary isom}  also implies that  $(\ker \acti) \varphi = \ker \actii$.
Since both $\ker \acti$ and $\ker \actii$ are finitely generated, checking this necessary condition is an instance of the \emph{Whitehead problem for subgroups in~$\Fn$}, denoted by $\WhPs(\Fn)$, which, on input two finitely generated subgroups $H,K \leqslant \Fn$, decides whether there exists an automorphism $\varphi \in \Aut(\Fn)$ such that $H \varphi = K$, and if so outputs one such automorphism. A variation of the classical Whitehead's peak reduction technique (used to prove the decidability of the homonym problem $\WhP(\Fn)$ for words), allowed Gersten to prove that $\WhPs(\Fn)$ is decidable as well (see~\cite{gersten_whiteheads_1984}).

Hence, if the answer of $\WhPs(\Fn)$ on the input $(\ker \acti, \ker \actii)$ is \nop, then the answer to our problem is \nop\ as well. Otherwise, $\WhPs(\Fn)$ outputs an automorphism $\varphi_0 \in \Aut(\Fn)$ such that $(\ker \acti) \varphi_0 = \ker \actii$. Note that then, for an arbitrary automorphism $\varphi \in \Aut(\Fn)$, $(\ker \acti) \varphi = \ker \actii$ if and only if $ \varphi \varphi_0^{-1} \in \stab(\ker \acti)$. That is, our set of candidate automorphisms $\varphi$ is reduced to $ \stab(\ker \acti) \varphi_0$, where $\stab(\ker \acti) = \set{\psi \in \Aut(\Fn) \st (\ker \acti)\psi = \ker \acti}$ is the (theoretically infinite) auto-stabilizer of $\ker \acti$. Note however that, since
both $\psi\in \stab(\acti)$ and $\acti$ factor through $\ker \acti$, and $(\ker \acti) \varphi = \ker \actii$, we can factorize the condition \eqref{eq: IP cond2} into the following diagram involving only finite (and computable) groups:
\begin{figure}[H]
\centering
\begin{tikzcd}
\Fn/\ker\acti \arrow[d, "\acti"', two heads] \arrow[r, "\overline{\psi}", dashed] & \Fn/\ker\acti \arrow[r, "\overline{\varphi}_0"] & \Fn/\ker\actii \arrow[d, "\actii", two heads] \\
\im \acti \arrow[rr, "{\inn[\matr{Q}]}", dashed]& & \im \actii                       
\end{tikzcd}
\end{figure}
Therefore, it is enough to check for every one of the finitely many $\overline{\psi} \in \overline{\stab}(\ker \acti)$ whether there exists some $\matr{Q} \in \GL_m (\ZZ)$ such that $\psi \varphi_0 \actii = \acti \inn[\matr{Q}]$.
But, as stated in \Cref{cor: MCP}, this amounts to finitely many instances of the $\MCP(\GL_m(\ZZ))$, and is therefore decidable in finite time. 
If at least one of these instances returns \yep\, then, the answer to our problem is \yep\ as well; otherwise the answer to our problem is \nop. This completes the proof.
\end{proof}

\subsection{Faithful actions}
The purpose of this section is to prove that the isomorphism problem restricted to FABF groups given by faithful actions coincides with a constrained version of the free subgroup conjugacy problem in $\GLm$, which we do not know if it is decidable. We remark it is not decidable whether a finite set of unimodular matrices of the same size generates a free group.



\begin{rem}
Let $m \in \NN$, Then, the problem
\begin{equation}
    \gen{\matr{A_1},\ldots,\matr{A_n}} \text{ is free} \?_{\subalign{&n\in\NN,\\ &\matr{A_1},\ldots,\matr{A_n} \in \GL_m(\ZZ)}}
\end{equation}
is decidable for $m=2$
and unknown for $m \geq 3$.
\end{rem}

We start by defining 
a constrained version of the generalized conjugacy problem.

\begin{defn} Let $G=\gen{X}$ be a group. The \emph{$n$-constrained free subgroup conjugacy problem} is
\begin{align}
    \CFSCP(G,n) \ \equiv\ 
    &\gen{S}_G  \conj \gen{T}_G
    \?\,_{S, T \subseteq_{\mathsf{finite}} (X^\pm)^* \,\st\, \gen{S}_G \,\isom\, \gen{T}_G \,\isom\, \Free[n]
     }
 \end{align}
\end{defn}



 \begin{lem}\label{lem: injective isomorphic}
 Let $m\geq 1, n\geq 2$, and let $\acti,\actii:\Fn\to \GL_m(\ZZ)$ be injective homomorphisms. Then $\Ga$ is isomorphic to $\Gb$ if and only if $\im(\acti)$ and $\im(\actii)$ are conjugate in $\GL_m(\ZZ)$. That is,
 \begin{equation}
     \Ga \isom \Gb
     \ \Leftrightarrow\
     \exists \matr{Q}\in \GL_m(\ZZ) \st \matr{Q}^{-1} \im(\acti) \matr{Q} = \im(\actii)
 \end{equation}
 \end{lem}
 \begin{proof}
 As seen in \Cref{lem: necessary isom} the implication to the right always holds. Conversely, if~$\im(\actii) = \matr{Q}^{-1} \im(\acti) \matr{Q} = \im(\acti) \inn[\matr{Q}]$ for some matrix $\matr{Q}\in \GL_m(\ZZ)$, then $\inn[\matr{Q}] \in \Inn(\GL_m(\ZZ))$ restricts to an isomorphism $\im(\acti) \to \im(\actii)$ and --- since $\acti$ and $\actii$ are, by hypothesis, injective --- the composition $\varphi_{\matr{Q}} = \acti \inn[\matr{Q}] \actii^{-1}$ is an automorphism of $\Fn$ obviously satisfying the condition $\varphi_{\matr{Q}} \actii=\acti\inn[\matr{Q}]$. By (\ref{eq: IP cond2}), $\Ga$ and $\Gb$ are isomorphic.
 \end{proof}

For every $m\geq 1, n\geq 2$, we denote by $\Fn \ltimes\inj \ZZ^m $ the family of $\ZZ^m$ by $\Fn$ groups with faithful action, and we denote by $\G\inj$ the whole family of (finitely generated) free-abelian by free groups with faithful action.
That is
\begin{align*}
    \Fn \ltimes\inj \ZZ^m
    &\,=\,
    \Set{
    \Fn \ltimes_{\acti} \ZZ^m \st \acti\colon \Fn \to \GL_m(\ZZ) \text{ is an injective homomorphism}
    }\\
    \G\inj
    &\,=\, 
    \textstyle{\bigcup_{m\geq 1, n\geq 2} \ \Fn \ltimes\inj \ZZ^m}
\end{align*}
Note that the isomorphism problem for groups belonging to the first family is:
\begin{align}
    \IP(\Fn \ltimes\inj \ZZ^m)
    &\ \equiv \ 
    \Ga \isom \Gb \?\,_{\acti,\actii\colon \Fn \to \GL_m(\ZZ) \text{ injective homomorphisms}}
\end{align}
where $\acti$ and $\actii$ can be assumed to be given as sets of images of the standard basis for $\Fn$; that is as respective $n$-tuples of matrices $(\matr{A_1},\ldots,\matr{A_n})$ and $ (\matr{B_1},\ldots,\matr{B_n})$ in $\GL_m(\ZZ)$ such that $\im (\acti) = \gen{\matr{A_1},\ldots,\matr{A_n}} \isom \gen{\matr{B_1},\ldots,\matr{B_n}} = \im (\actii) \isom \Fn$.

 

\begin{prop}
For every $m\geq 1$ and every $n\geq 2$,
the decidability of $\IP(\Fn \ltimes\inj \ZZ^m)$ is equivalent to the decidability of $\CFSCP(\GL_m(\ZZ),n)$; that is,
\begin{equation}
    \IP(\Fn \ltimes\inj \ZZ^m)
    \ \equiv_T \ 
    \CFSCP(\GL_m(\ZZ),n)
\end{equation}
\end{prop}

 \begin{proof}
 Suppose that~$\CFSCP(\GL_m(\ZZ),n)$ is decidable and
let $\acti,\actii$ be given injective homomorphisms $\Fn\to \GL_m(\ZZ)$. That is,
 we are given
two $n$-tuples
 $(\matr{A_1},\ldots,\matr{A_n})$ and $(\matr{B_1},\ldots,\matr{B_n})$ of matrices in $\GL_m(\ZZ)$ generating the subgroups $\im (\acti) \isom \Fn$ and $\im(\actii) \isom \Fn$ respectively. 
  According to \Cref{lem: injective isomorphic}, in this situation, $\Ga \isom \Gb$ if and only if $\im(\acti) = \gen{\matr{A_1},\ldots,\matr{A_n}}$ is conjugate to $\im(\actii) = \gen{\matr{B_1},\ldots,\matr{B_n}}$, which is precisely the output of $\CFSCP(\GL_m(\ZZ),n)$ on input $\matr{A_1},\ldots,\matr{A_n};\matr{B_1},\ldots,\matr{B_n}$.
Hence, $\IP(\Fn \ltimes\inj \ZZ^m)
     \leq_T  
    \CFSCP(\GL_m(\ZZ),n)$.

Conversely, 
suppose that $\IP(\Fn \ltimes\inj \ZZ^m)$ is decidable and let $\set{\matr{C_1},\ldots,\matr{C_p}}$ and $\set{\matr{D_1},\ldots,\matr{D_q}}$ be two sets of matrices in $\GL_m(\ZZ)$ generating respective free subgroups of rank $n$.
Then, since the rank is known by assumption, 
we can compute bases $(\matr{A_1},\ldots,\matr{A_n})$ for $\gen{\matr{C_1},\ldots,\matr{C_p}}$ and $(\matr{B_1},\ldots,\matr{B_n})$ for $\gen{\matr{D_1},\ldots,\matr{D_q}}$.
Now it is enough to consider the (injective) homomorphisms $\acti, \actii \colon \Fn \to \GL_m(\ZZ)$ given by $x_i \to \matr{A_i}$ and $\actii: x_i \mapsto \matr{B_i}$ (for $i \in [1,n]$), and give them as an input for $\IP(\Fn \ltimes\inj \ZZ^m)$. Again by \Cref{lem: injective isomorphic} this procedure will decide whether $\gen{\matr{A_1},\ldots,\matr{A_n}} = \gen{\matr{C_1},\ldots,\matr{C_p}}$ and 
$\gen{\matr{B_1},\ldots,\matr{B_n}} = \gen{\matr{D_1},\ldots,\matr{D_q}}$ are conjugate in $\GL_m(\ZZ)$.
Hence,
$\CFSCP(\GL_m(\ZZ),n)
\leq_T  
\IP(\Fn \ltimes\inj \ZZ^m)$, and the proof is complete.
 \end{proof}
 
 

\section{Orbit problems and applications}\label{sec: orbit problems}
We generically call orbit decision problems (\OD\ problems) to
the problems of deciding whether two given objects belong to the same orbit under the action of a certain family
$\mathcal{T}$ of transformations. If an algorithm to answer such question exists, then we say that the family $\mathcal{T}$ is orbit decidable (or that the problem $\OD(\mathcal{T})$ is decidable).

In the context of group theory,  we usually consider a fixed group $G$ (belonging to a certain class), the family of transformations is some subfamily $\mathcal{T} \subseteq \End(G)$, and the objects are typically (conjugacy classes of, tuples of) elements or subgroups in $G$. We start recalling that this is a broad schema that embraces many classical problems in group theory. 
For example, the second of Dehn's problems --- the \emph{conjugacy problem} $\CP(G)$ --- is nothing more than ${\OD(\Inn(G))}$. Other well-known orbit problems, usually referred to as \emph{Whitehead problems} in the literature, correspond to the cases where $\mathcal{T}$ is the full set of endomorphisms, monomorphisms or automorphisms of $G$ (as well as their variants modulo conjugation). For example, in \cite{whitehead_equivalent_1936}, \citeauthor{whitehead_equivalent_1936} introduces the seminal peak-reduction technique to solve $\OD(\Aut(\Fn))$,
which has been subsequently generalized to solve different variants of the problem (dealing with subgroups, tuples, etc.). Also in the realm of free groups, the orbit decidability of $\Mon(\Fn)$ was proved by Ciobanu and Houcine in \cite{ciobanu_monomorphism_2010}, whereas that of $\End(\Fn)$ follows from the work of Makanin in~\cite{makanin_equations_1982}. 

\subsection{Brinkmann's Problems}
In \cite{brinkmann_detecting_2010}, Brinkmann considers cyclic subgroups of automorphisms of the free group and proves their uniform orbit decidability both for words and conjugacy classes. Following Brinkmann's initial research, we  call \defin{Brinkmann problems} the problems about orbit decidability of cyclic subgroups of endomorphisms.

\begin{defn} \label{def: Brinkmann problems}
Let $G$ be a group
(given by a finite presentation $\pres{X}{R}$)
and let $\mathcal{T} \subseteq \End(G)$ be a family of endomorphisms
(given as a set of words representing the images of the generators).
The \defin{Brinkmann problem} 
in $G$ \wrt $\mathcal{T}$, denoted by $\BrP_{\mathcal{T}}(G)$
is
the problem of deciding,
given an endomorphism $\varphi \in \mathcal{T}$ and two elements $x,y \in G$,
whether there exists some $n \in \NN$ such that $(x)\varphi^n = y$; \ie
\begin{equation} \label{eq: BrP}
\BrP_{\mathcal{T}}(G)
\,\equiv\,
\exists n \in \NN \st (x) \varphi^n = y \?
_{\subalign{&\varphi\in\mathcal{T},\\ & x,y\in G}}
\end{equation}
Similarly, the \defin{Brinkmann conjugacy problem} in $G$ \wrt $\mathcal{T}$, denoted by $\BrCP(G)$ is
the problem of deciding,
given an endomorphism $\varphi \in \mathcal{T}$ and two elements $x,y \in G$,
whether there exists some $n \in \NN$ such that $(x)\varphi^n$ and $y$ are conjugate in $G$; \ie
\begin{equation} \label{eq: BrCP}
\BrCP_{\mathcal{T}}(G)
\,\equiv\,
\exists n \in \NN \st (x) \varphi^n \conj y \?
_{\subalign{&\varphi\in\mathcal{T},\\ & x,y\in G}}
\end{equation}
We abbreviate $\BrPa(G) = \BrP_{\!\Aut(G)}(G)$, $\BrPm(G) = \BrP_{\Mon(G)}(G)$, and $\BrPe(G) = \BrP_{\End(G)}$, and similarly for $\BrCP$.
\end{defn}

\begin{rem} \label{rem: BP(id) = WP}
Note that $\BrP_{\mathrm{id}_G}(G)  = \WP(G)$ and   $\BrCP_{\mathrm{id}_G}(G) =\CP(G)$. 
Hence, it is clear that
$\WP(G) \preceq \BrPa(G) \preceq \BrPm(G) \preceq \BrPe(G)$ and 
$\CP(G) \preceq \BrCPa(G) \preceq \BrCPm(G) \preceq \BrCPe(G)$.
\end{rem}

\begin{rem} \label{rem: yes outputs}
Note that the decidability of Brinkmann's problems
immediately allows us to compute a witness in case it exists: if the answer is \yep, then it is enough to keep enumerating the successive images $((x)\varphi^n)_{n \geq 0}$ and use the positive part of word problem (\resp conjugacy problem) to check for a guaranteed match with $y$. In fact, we will see that if $\WP(G)$ (\resp $\CP(G)$) is decidable
then we can compute the full set of witnesses for Brinkmann's problems in the sense made precise below. 
\end{rem}

\begin{defn}
Let $G$ be a group, let $x,y \in G$, and let $\varphi \in \End (G)$. Then, the set of \defin{$\varphi$-logarithms} (\resp \defin{$\widetilde{\varphi}$-logarithms}) of $y$ in base $x$ is
\begin{equation*}
    \varphi\text{-\!}\log_x(y)
    \,=\,
    \set{k\geq 0 \st (x)\varphi^k = y}
    \qquad
    \text{(\resp}\widetilde{\varphi}\text{-\!}\log_x(y)
    \,=\,
    \set{k\geq 0 \st (x)\varphi^k \conj y}\text{)}
\end{equation*}

Note that $0 \in \varphi\text{-\!}\log_x(y)$ 
(\resp $0 \in \widetilde{\varphi}\text{-\!}\log_x(y)$)
if and only if
$x=y$ (\resp $x\conj y$); and then
$
\varphi\text{-\!}\log_x(x) = p\NN$ 
(\resp $
\widetilde{\varphi}\text{-\!}\log_x(x) = p\NN$), for some $p \in \NN$, which is called the 
\defin{$\varphi$-period} of~$x$
(\resp
\defin{$\widetilde{\varphi}$-period} of~$x$).
In particular we say that the $\varphi$-period
(\resp $\widetilde{\varphi}$-period)
of $x$ is $0$ if~$(x)\varphi^k \neq x$
(\resp $(x)\varphi^k \nconj x$) for all $k>0$. 
\end{defn}


\begin{lem} \label{lem: philog form}
Let $G$ be a group, let $x,y \in G$, and let $\varphi \in \End (G)$. Then, 
\begin{enumerate}[ind]
    \item \label{item: philog}
    either $\varphi\text{-\!}\log_x(y) = \varnothing$ or $\varphi\text{-\!}\log_x(y) = k_0 + p\NN$,
    where $k_0 = \min (\varphi\text{-\!}\log_x(y))$, and $p$ is the $\varphi$-period of $y$.
    \item \label{item: cphilog}
    either $\widetilde{\varphi}\text{-\!}\log_x(y) = \varnothing$ or $\widetilde{\varphi}\text{-\!}\log_x(y) = k_0 + p\NN$,
    where $k_0 = \min (\widetilde{\varphi}\text{-\!}\log_x(y))$, and $p$ is the $\widetilde{\varphi}$-period of $y$.
\end{enumerate}
\end{lem}

\begin{proof}
\ref{item: philog} Assume that $\varphi\text{-\!}\log_x(y) \neq \varnothing$ and let $k_0 = \min (\varphi\text{-\!}\log_x(y))$ and $p$ be the $\varphi$-period of~$y$. Then, the inclusion to the left is clear from the definitions.
Conversely, suppose that $k \in \varphi\text{-\!}\log_x(y)$, then $(x)\varphi^k = y = (x)\varphi^{k_0}$, where $k \geq k_0$.
That is, $(y) \varphi^{k-k_0} = ((x)\varphi^{k}) \varphi^{k-k_0} = (x)\varphi^{k_0} = y$ and hence $k-k_0 \in \varphi$-$\log_y(y) = p \NN$ as claimed. The proof for \ref{item: cphilog} is exactly the same changing equality by conjugacy and $\varphi$ by $\widetilde{\varphi}$.
\end{proof}

Note that the Brinkmann problems \eqref{eq: BrP} and \eqref{eq: BrCP} consist in nothing more than deciding about the emptiness of the sets of logarithms $\varphi\text{-\!}\log_x(y)$ and $\widetilde{\varphi}\text{-\!}\log_x(y)$ respectively. Below we prove that these decisions are indeed algorithmically equivalent to the explicit computation of the sets of logarithms according their description in \Cref{lem: philog form}.

\begin{lem} \label{lem: compute logs}
\begin{enumerate*}[ind]
    \item \label{item: compute philog}
    If $\WP(G)$ and $\BrP_{\!\varphi}(G)$ are decidable then, given $x,y \in G$, we can
compute
$\varphi\text{-\!}\log_x(y)
$.
    \item \label{item: compute cphilog}
    If $\CP(G)$ and $\BrCP_{\!\varphi}(G)$ is decidable then, given $x,y \in G$, we can decide whether $\widetilde{\varphi}\text{-\!}\log_x(y)$
is empty and, if not, compute $k_0$ and $p$ such that $\widetilde{\varphi}\text{-\!}\log_x(y) = k_0 + p\NN$.
\end{enumerate*}

\end{lem}

\begin{proof}
\ref{item: compute philog}
Assume that both $\WP(G)$ and $\BrP_{\!\varphi}(G)$ are decidable. If, on input $(x,y) \in G^2$, $\BrP_{\!\varphi}(G)$ answers \nop, then
$\varphi\text{-\!}\log_x(y) = \varnothing$. Otherwise, 
there exists some $k \geq 0$ such that $(x) \varphi^k = y$. Now it is enough 
use the decidability of the word problem to successively check (starting from $n=0$) whether $(x)\varphi^n = y$ until a guaranteed first match occurs, precisely at step $n = k_0$. Similarly, in order to compute the $\varphi$-period $p$ of $y$, we run  $\BrP_{\!\varphi}(G)$ on input $((y)\varphi, y)$ to check whether there is any $k \geq 1$ such that $(y)\varphi^k = y$. If the answer is \nop, then $\varphi\text{-\!}\log_y(y) = \set{0}$ and hence $p = 0$. Otherwise, we know that $\varphi\text{-\!}\log_y(y) \supsetneq \set{0}$, and we can use the decidability of $\WP(G)$ in the same way as before to compute the least element in $\varphi\text{-\!}\log_y(y) \setmin \set{0}$, which is precisely $p$. The proof for~\ref{item: compute cphilog} is completely analogous, changing $\WP(G)$ by $\CP(G)$, $\BrP_{\!\varphi}(G)$ by $\BrCP_{\!\varphi}(G)$, equality by conjugacy, and $\varphi$ by~$\widetilde{\varphi}$.
\end{proof}




In this section we study Brinkmann's problems in the context of FABF groups $\Ga$. From \Cref{rem: BP(id) = WP} and the undecidability of $\CP(\FABF)$ proved in \cite{bogopolski_orbit_2010} it immediately follows the undecidability of $\BrCP(\FABF)$ \wrt any family of endomorphisms containing the identity. Hence we will focus in $\BrP(\FABF)$.
A natural starting point is to consider the same problem(s) on its factors (namely, finitely generated free and free-abelian groups). 

For free-abelian groups, the relevant result is stated in  \cite{kannan_polynomial-time_1986}, where \citeauthor{kannan_polynomial-time_1986} provide  a polynomial algorithm to solve $\BrPe$ in the groups of the form $\QQ^m$. We state their result below, restricted to $\Zm$.

\begin{thm}[\citenr{kannan_polynomial-time_1986}]
\label{thm: KannanLipton}
The Brinkmann problem $\BrPe(\Zm)$ is algorithmically decidable. 
\end{thm}

\begin{rem} \label{rem: Kannan affine}
We recall that every affine transformation can be represented as a (restriction of) a linear transformation; namely, for every $\vect{b},\vect{x},\vect{y} \in \Zm$ and every $\matr{P} \in \Mn$,
\begin{equation*}
\vect{x} \matr{P} + \vect{b} = \vect{y}
\ \Leftrightarrow \
\setlength\arraycolsep{3pt}
\begin{pmatrix}
\vect{x} & 1
\end{pmatrix}
\begin{pmatrix}
\matr{P} & \vect{0}\\
\matr{b} & 1
\end{pmatrix}
=
\begin{pmatrix}
\vect{y} & 1
\end{pmatrix}
\,.
\end{equation*}
Hence,
\Cref{thm: KannanLipton} immediately provides the decidability of the Brinkmann problem for affine transformations of $\ZZ^m$ as well. 
\end{rem}

Regarding the free group, the initial work of Brinkmann has recently found continuation in the preprint \cite{logan_conjugacy_2023}, where A.~Logan uses Mutanguha's techniques extending to endomorphisms the computability of the fixed subgroup of endomorphisms of $\Fn$ (see~\cite{mutanguha_constructing_2022}) to prove the decidability of \CP\ for ascending HNN-extensions of free groups. On his way to this result, Logan considers and solves several variants of the Brinkmann Problem
 from which he deduces their decidability for monomorphisms of the free group. Extending a bit further Logan's analysis the decidability for general endomorphisms was proven in \cite{carvalho_decidability_2023}.

\begin{thm} \label{thm: BrP free}
Let $\Fn$ be a finitely generated free group of rank $n$. Then, the Brinkmann problems $\BrPe(\Fn)$ and $\BrCPe(\Fn)$ are decidable. \qed
\end{thm}

Below, we combine \Cref{thm: KannanLipton,thm: BrP free} with our description of powers of endomorphisms of type I 
to prove the decidability of $\BrP$ for this latest family. 
The auxiliary lemma below provides an affine recursion which is a crucial ingredient in the proof.

\begin{lem}\label{lem: powers periodic affine}
Let $\Ga$ be a $\FABF$ group, $\Phi = \Phi_{\varphi,\matr{Q},\matr{P}}\in \Endi(\Ga)$, and $k,p\in \NN$ and $u\in \Fn$ be such that
$u \varphi^{k + p} = u \varphi^k$.
Then, for all $\lambda\geq 1$ and $\vect{a}\in \Zm$, 
\begin{equation} \label{eq: powers periodic affine}
(u\ta{a})\Phi^{k+\lambda p}\tau=\left((u\ta{a})\Phi^k\tau \right)\Theta^\lambda
\end{equation}
where $\Theta:\Zm\to\Zm$ is the affine transformation $\vect{x}\mapsto \vect{x
}\matr{Q}^p
+ u \varphi^{k} \, \bbeta_{\Phi^p}
$
\end{lem}

\begin{proof}
Using the recurrences \eqref{eq: EndI tau} and \eqref{eq: EndI^k tau rec},
and taking into account
that $u\varphi^{k + p} = u\varphi^{k}$ we obtain:
\begin{equation}
\begin{aligned} \label{eq: periodic affine}
(u\ta{a})\Phi^{k + \lambda p}\tau
&\,=\,
(u\ta{a})\Phi^{k + (\lambda-1) p} \Phi^{p}\tau\\
&\,=\,
(u\ta{a})\Phi^{k + (\lambda-1) p} \tau \,\matr{Q}^p +
u \varphi^{k+ (\lambda - 1) p} \, \bbeta_{\Phi^p}\\
&\,=\,
(u\ta{a})\Phi^{k + (\lambda-1) p}  \tau \,\matr{Q}^p +
u \varphi^{k} \, \bbeta_{\Phi^p}
\end{aligned}
\end{equation}
and the claimed result follows.
\end{proof}

\begin{thm} \label{thm: BrP typeI}
Let $\Ga$ be a FABF group. Then, the Brinkmann problem $\BrP_{\Endi}(\Ga)$ is decidable.
That, is
there exists an algorithm that, given $u\ta{a},v\ta{b} \in \Ga$ and $\Phi \in \Endi(\Ga)$, 
decides whether there exists some $k \in \NN$ such that $(u \ta{a}) \Phi^k = v \ta{b}$.

In particular, $\BrPa(\Ga)$ and $\BrPm(\Ga)$ are decidable. 
\end{thm}

\begin{proof}
Recall that type I endomorphisms of $\Ga$ are of the form $\Phi_{\varphi,\matr{Q},\matr{P}}$ in \eqref{eq: type I}, where ${\varphi\in \End(\Fn)}$, $\matr{Q} \in \mathcal{M}_{m}(\ZZ)$
and $\matr{P} \in \mathcal{M}_{n \times m}(\ZZ)$ satisfying
$
    \matr{Q} \matr{A}_{x\! \varphi}
    = \matr{A}_x \matr{Q}$,
for all $ x\in \Fn$. 
So, let $\matr{A}_1,\ldots \matr{A}_n \in \GL_m(\ZZ)$ be unimodular matrices (defining the semidirect action $\acti$), and suppose we are given such $(\varphi,\matr{Q},\matr{P})$
and two elements $u\ta{a},v\ta{b}\in \Ga$.
We want to decide whether there exists some $k\in \NN$ such that $(u \ta{a}) \Phi^k = v \ta{b}$, or equivalently --- 
splitting the free and free-abelian parts --- such that the following two conditions hold:
\begin{align}[left=\empheqlbrace]
    \ (u) \varphi^k \,=\, (u\ta{a}) \Phi^k \pi &\,=\, v \label{eq: }\\\hspace{2pt}
    \ (u\ta{a}) \Phi^k \tau 
    &\,=\, \vect{b} \label{eq: BrP ab}
\end{align}

Note that the restriction to the first condition is just
an instance of $\BrPe(\Fn)$, which is decidable by \Cref{thm: BrP free}.(c). Hence, if $\BrPe(\Fn)$ answers \nop\ on input $(u,v,\varphi)$
(that is, there is no $k \in \NN$ such that $u \varphi^k = v$)
then the answer to our problem is also \nop, and we have finished. Otherwise, (since $\WP(\Ga)$ is decidable) from \Cref{lem: compute logs}, 
we can assume that $\BrPe(\Fn)$ outputs $k_0 = \min\set{k\in \NN \st u\varphi^k=v}$ and $p\in \NN$ such that $\varphi\text{-\!}\log_u(v) = k_0 + p\NN$.

If $p=0,$ then $k_0$ is the only candidate solution for our problem. So, we simply compute $(u\ta{a})\Phi^{k_0}$ and verify whether $(u\ta{a})\Phi^{k_0} = v\ta{b}$. If so, then our answer is $\yep$ and otherwise our answer is~\nop.

It only remains to decide the case where $p \geq 1$.
In this case, the set of solutions for $u \varphi^k = v$ is precisely $\varphi\text{-\!}\log_u(v) = k_0 + p\NN$. So it is enough to decide whether one of them satisfies \eqref{eq: BrP ab}, that is whether
\begin{equation} \label{eq: BrPa tau}
    \exists \lambda \in \NN \st (u\ta{a}) \Phi^{k_0 + \lambda p} \tau \,=\, \vect{b} \,,
\end{equation} 
which according to \Cref{lem: powers periodic affine},  can be rewritten as:
\begin{equation}
\exists \lambda \in \NN \st ((u\ta{a}) \Phi^{k_0} \tau) \, \Theta^{\lambda} \,=\, \vect{b} ,
\end{equation}
which, in turn, is just an instance of the Brinkmann problem for affine transformations, and hence decidable by \Cref{thm: KannanLipton} and \Cref{rem: Kannan affine}. This completes the proof.
\end{proof}

\begin{rem}\label{rem: BrP yes BrCP no}
The classical question of whether a group with solvable word problem would necessarily have solvable conjugacy problem was answered negatively in \cite{miller_iii_group-theoretic_1971}. Since~$\CP(\FABF)$ (and hence $\BrCP(\FABF)$) is known to be unsolvable, it follows from the last theorem that the family of FABFs contains examples of groups with solvable $\BrP$ but unsolvable $\BrCP$.
\end{rem}

\section*{Acknowledgements}
The first author was supported by
CMUP, which is financed by national funds through FCT – Fundação
para a Ciência e a Tecnologia, I.P., under the projects with reference UIDB/00144/2020
and UIDP/00144/2020.

The second author acknowledges support from the Spanish \emph{Agencia Estatal de Investigación} through grant PID2021-126851NB-I00 (AEI/FEDER, UE), as well as from the \emph{Universitat Politècnica de Catalunya} in the form of a ``María Zambrano'' scholarship.


\renewcommand*{\bibfont}{\small}
\printbibliography
\Addresses


\end{document}